\documentclass[12pt]{article}
 \usepackage[margin=1in]{geometry} 
\usepackage{amsmath,amsthm,amssymb,amsfonts}
\usepackage{graphicx}
\usepackage{url}
\usepackage{float} 

\newcommand{\R}{\mathbb{R}}

\theoremstyle{plain}
\newtheorem{theorem}{Theorem}
\newtheorem{lemma}[theorem]{Lemma}

\newtheorem{corollary}[theorem]{Corollary}
\theoremstyle{definition}

\begin{document}
 

\title{A Lower Bound on The Blob}
\author{Jamie Schmidt}
\maketitle

\begin{abstract}

A blob grows on the plane (as in films called ``The Blob" of 1958 and 1988).  Suppose that it grows in all
directions at unit rate, and can be stopped only by a certain kind of fence that can be manufactured
at rate $\lambda$.  What is the critical value $\lambda_C$ for this rate, above which the blob can be
surrounded and contained, and below which we are all doomed?

This problem, introduced by Bressan \cite{bressan2007} in 2007 and independently by Barghi and Winkler \cite{BW} in 2013,
was in both cases motivated by consideration of strategies in fighting forest fires.  In particular, does it
pay to place ``gambit" barriers which will subsequently be overwhelmed by the fire, in order to slow it
down enabling later containment?

A non-gambit strategy, which will be described, succeeds when $\lambda > 2$ and is widely believed to be
optimal; but up until now no one has succeeded in obtaining a lower bound greater than 1 for $\lambda_C$.
We introduce a novel weakened version of the blob which is used to raise the lower bound to 1.5.

\end{abstract}

\newpage
\section{Introduction}
\label{sec:upper-bound}

The blob acts in the following manner. At time \(t=0\), it begins as the unit disk. 

\begin{figure}[H]
    \centering
    \includegraphics[ width=0.5\linewidth]{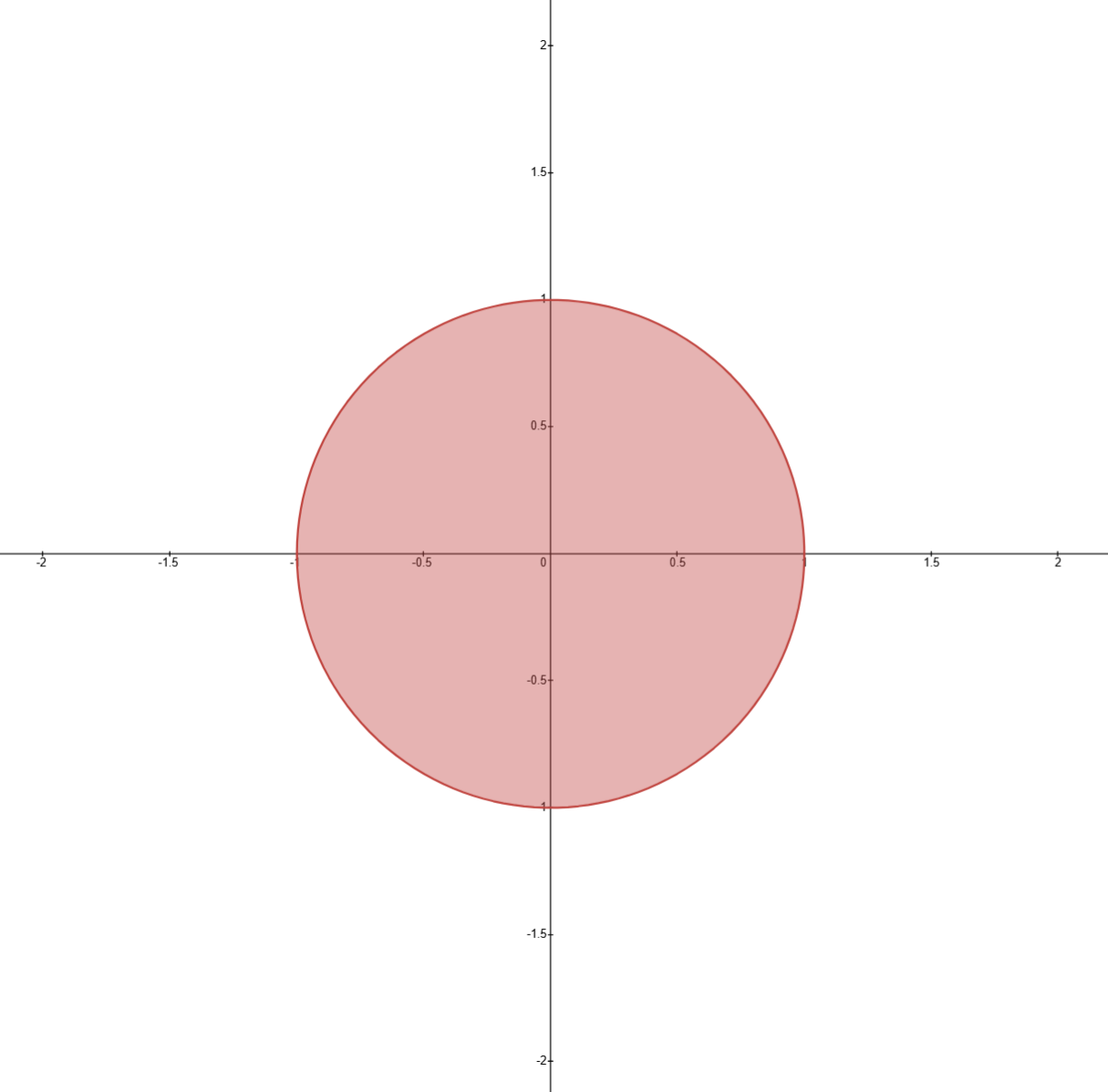}
    \caption{The blob at time zero.}
    \label{fig:the-blob-at-time-zero}
\end{figure}

The blob then grows at a rate of one unit per second.

\begin{figure}[H]
    \centering
    \includegraphics[width=0.5\linewidth]{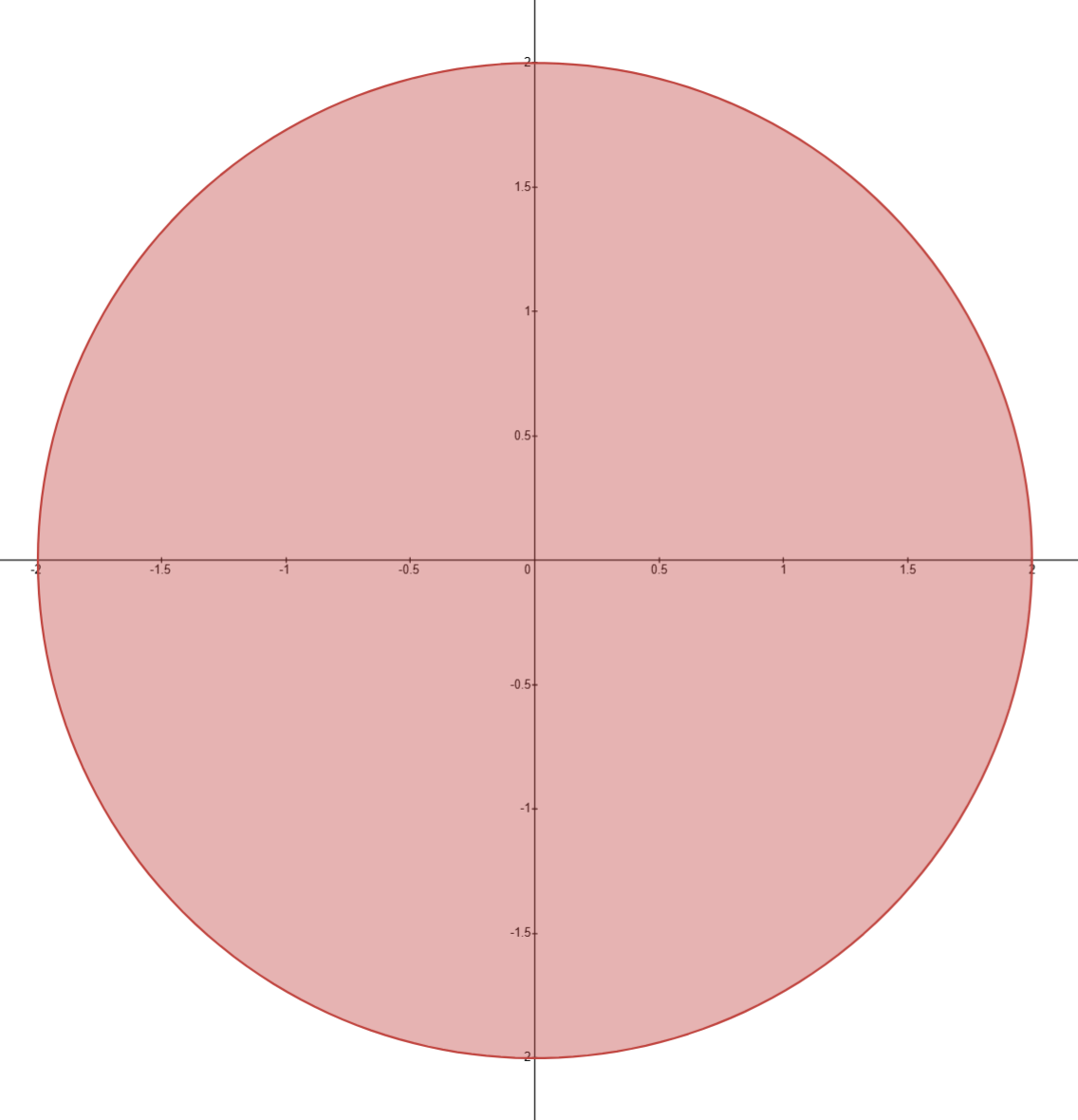}
    \caption{The blob at time one.}
    \label{fig:the-blob-at-time-one}
\end{figure}

If there is a wall in the way, as the blob grows from all points, it begins to encircle the wall. 

\begin{figure}[H]
    \centering
    \includegraphics[width=0.75\linewidth]{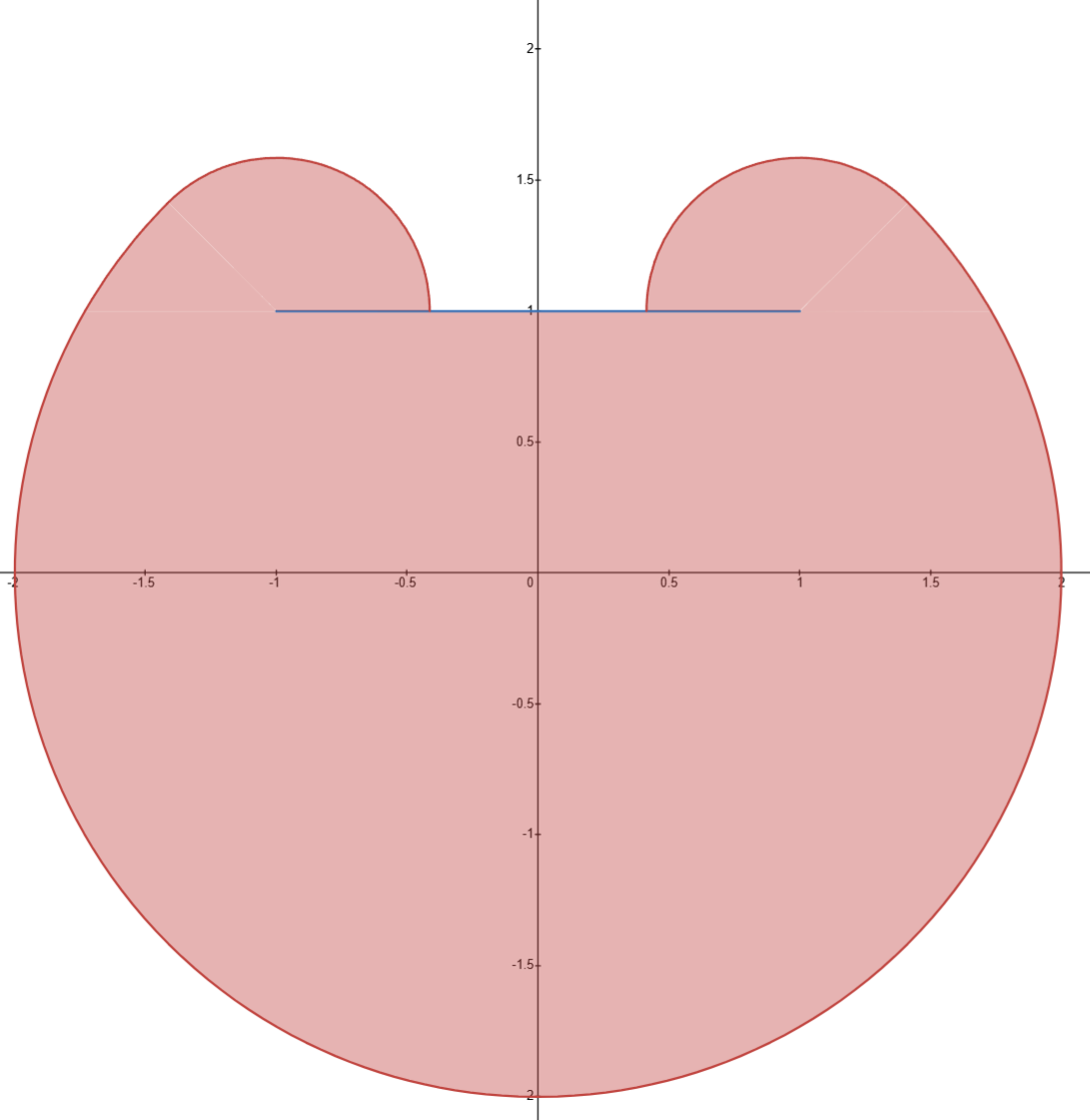}
    \caption{The blob at time one, with a wall drawn from (-1,1) to (1,1).}
    \label{fig:the-blob-at-time-one-with}
\end{figure}

The walls are defined as a finite union of rectifiable, compact curves, the blob at time \(t\) is the set of all points with shortest-path distance \(t+1\) from the origin, the paths in question being the set of all paths which do not intersect a wall.  The containment is successful if the blob is contained within a closed figure at some time \(T\) and if, at all times \(t\), the length of wall within shortest-path distance \(t+1\) from the origin is at most \(\lambda t\). 

We now note that, by multiplying both distance and time by a constant factor, the rates of 1 and \(\lambda\) remain unchanged, as a unit per second is the same as \(x\) units per \(x\) seconds. This means that we can effectively ``shrink" the blob by an arbitrary factor. 

\begin{figure}[H]
    \centering
    \includegraphics[width=0.75\linewidth]{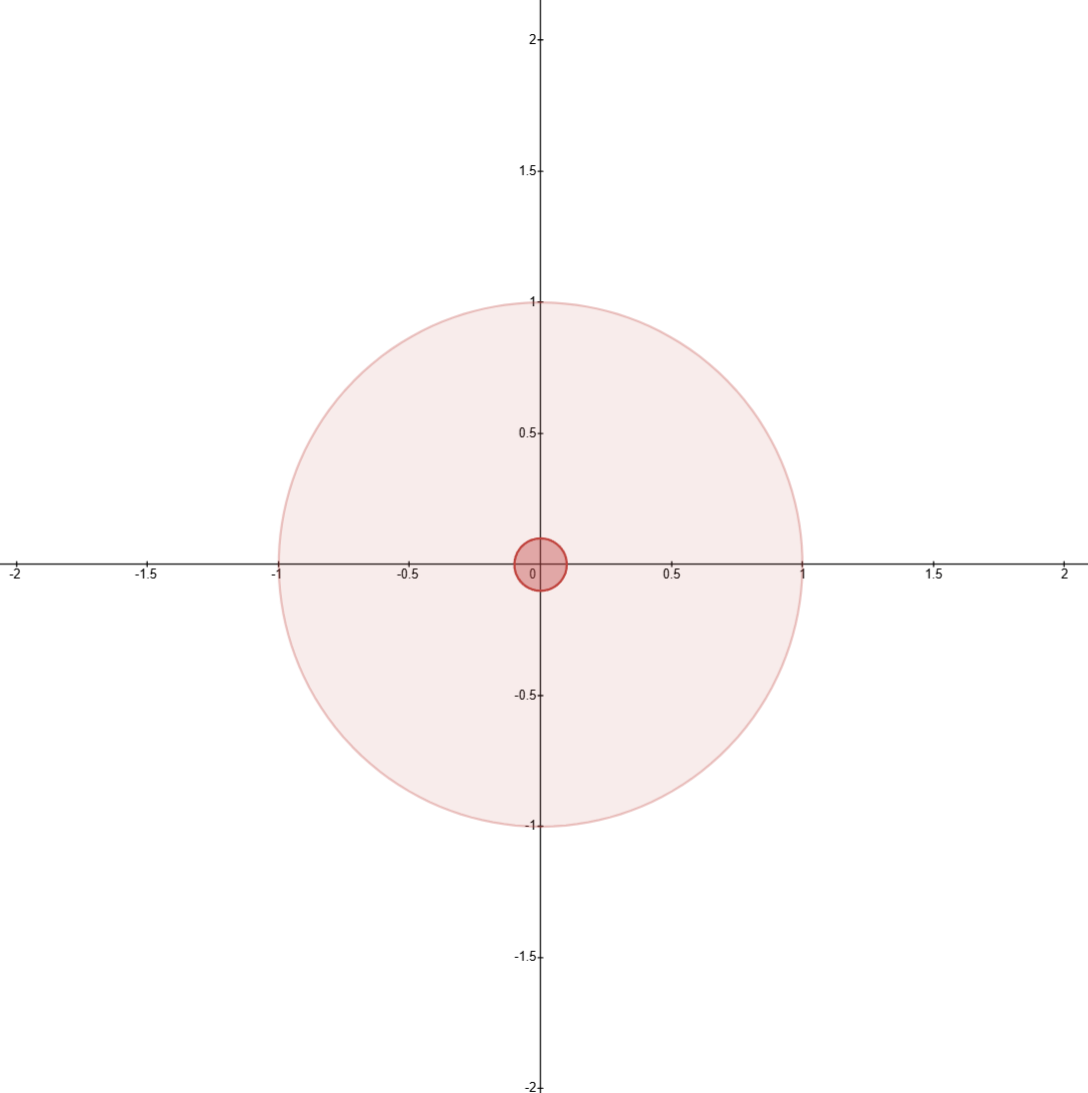}
    \caption{The blob shrunk by a factor of 10. This will also, effectively, shrink the time constant by 10.}
    \label{fig:the-blob-shrunk-by-a-factor}
\end{figure}

\subsection{Related work}

The blob problem, in this generality, is Bressan's dynamic blocking problem for a model of forest fire propagation~\cite{bressan2007,bressan2008}, and the bounds \(2\geq \lambda_C\geq 1\), reproved in Section~\ref{sec:upper-bound} and Section~\ref{sec:vine}, are originally due to Bressan. Existence of optimal blocking strategies was established by Bressan and De Lellis~\cite{bressandelellis2009}. The general question of whether \(\lambda_C=2\), Bressan's Fire Conjecture, remains open; a standing prize has been offered for its resolution~\cite{bressanprize}. Most of Bressan's work on this problem was independently replicated by Barghi and Winkler \cite{BW}, who reframed the problem in terms of a blob. Later, in unpublished work, Winkler and Schmidt independently replicated most of the rest of Bressan's work on this problem. 

Restricted variants of the problem have seen more progress: for barriers built as a self-closing spiral around the fire, the sharp threshold \(\lambda\approx2.6144\) is known~\cite{kleinlangetepelevcopoulos2015,bianchinizizza2025}, matching our own independent work, and for fire confined to a half-plane by a horizontal barrier with vertical delaying segments, Kim, Klein, K\"ubel, Langetepe, and Schwarzwald showed \(\lambda>1.66\) is necessary and \(\lambda=1.8772\) is sufficient~\cite{kimetal2019}. As far as we are aware, no improvement to the general lower bound of \(\lambda_C\geq1\) for the original, unrestricted problem has previously appeared. The main contribution of this paper is to push that bound to \(\lambda_C\geq1.5\).

\subsection{Prior Results on the Blob}

In our pursuit of an improved lower bound, we may slightly weaken the blob by assuming the limit. We now assume the blob starts at a single point, with an arbitrary ``exclusion radius" around it in which we cannot create walls. This radius is truly arbitrary, as it may always be assumed to be smaller than the closest any wall gets to the single point, as long as no wall gets arbitrarily close to said point. 

\begin{figure}[H]
    \centering
    \includegraphics[width=0.75\linewidth]{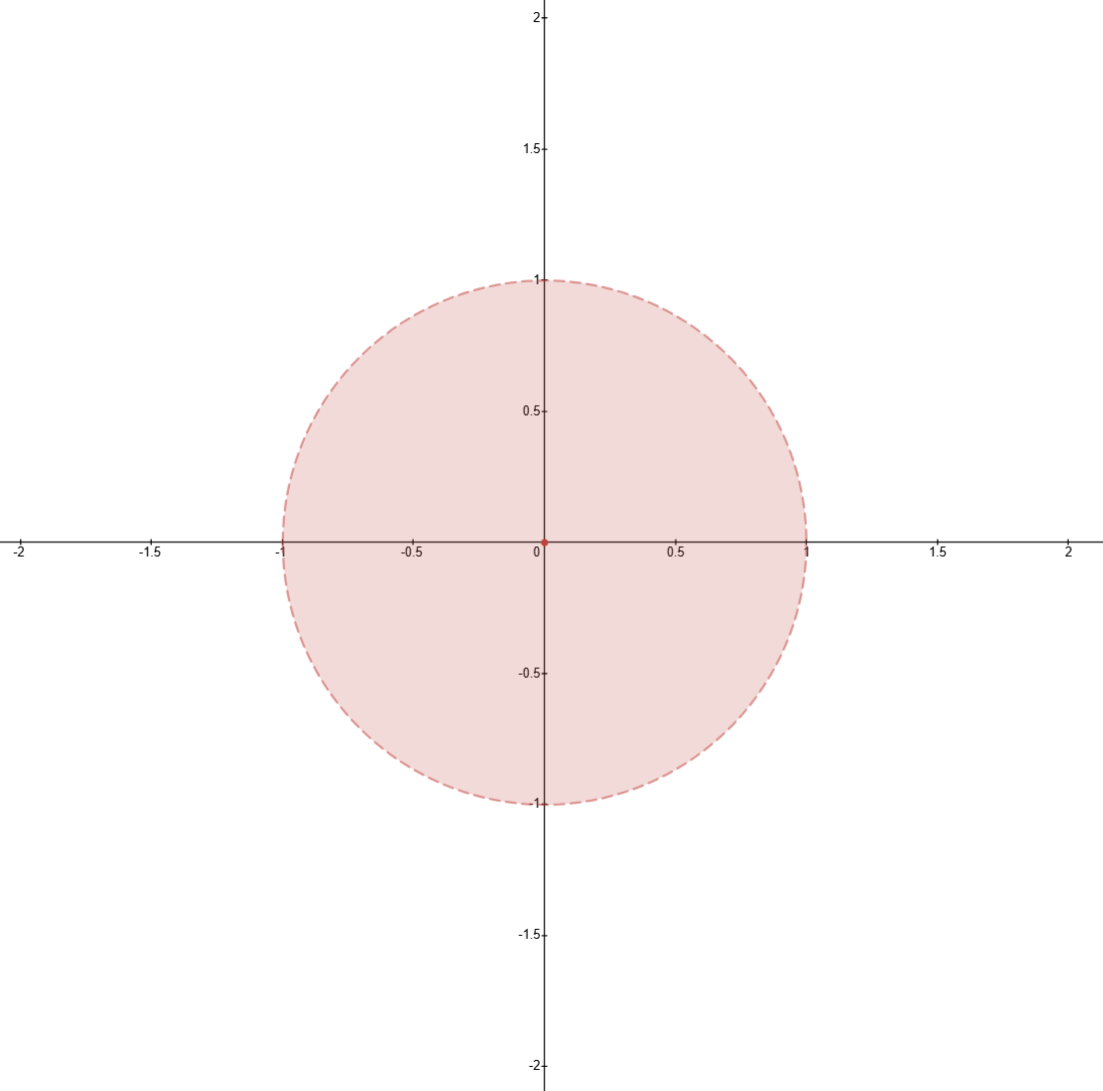}
    \caption{The blob as a point with an exclusion radius of 1.}
    \label{fig:the-blob-as-a-point-with}
\end{figure}

\begin{theorem}[Bressan~\cite{bressan2007}]
\label{thm:upper-bound}
The blob can be contained for any \(\lambda>2\).
\end{theorem}

\begin{proof}
We give the following ``teardrop" method: the blob may be contained with \(\lambda=2+\epsilon\) for any \(\epsilon>0\).

We begin by doing nothing until time \(t=\delta\), where \(\delta>0\) is the exclusion radius. 

\begin{figure}[H]
    \centering
    \includegraphics[trim=10cm 10cm 10cm 10cm, clip=true, width=0.4\linewidth]{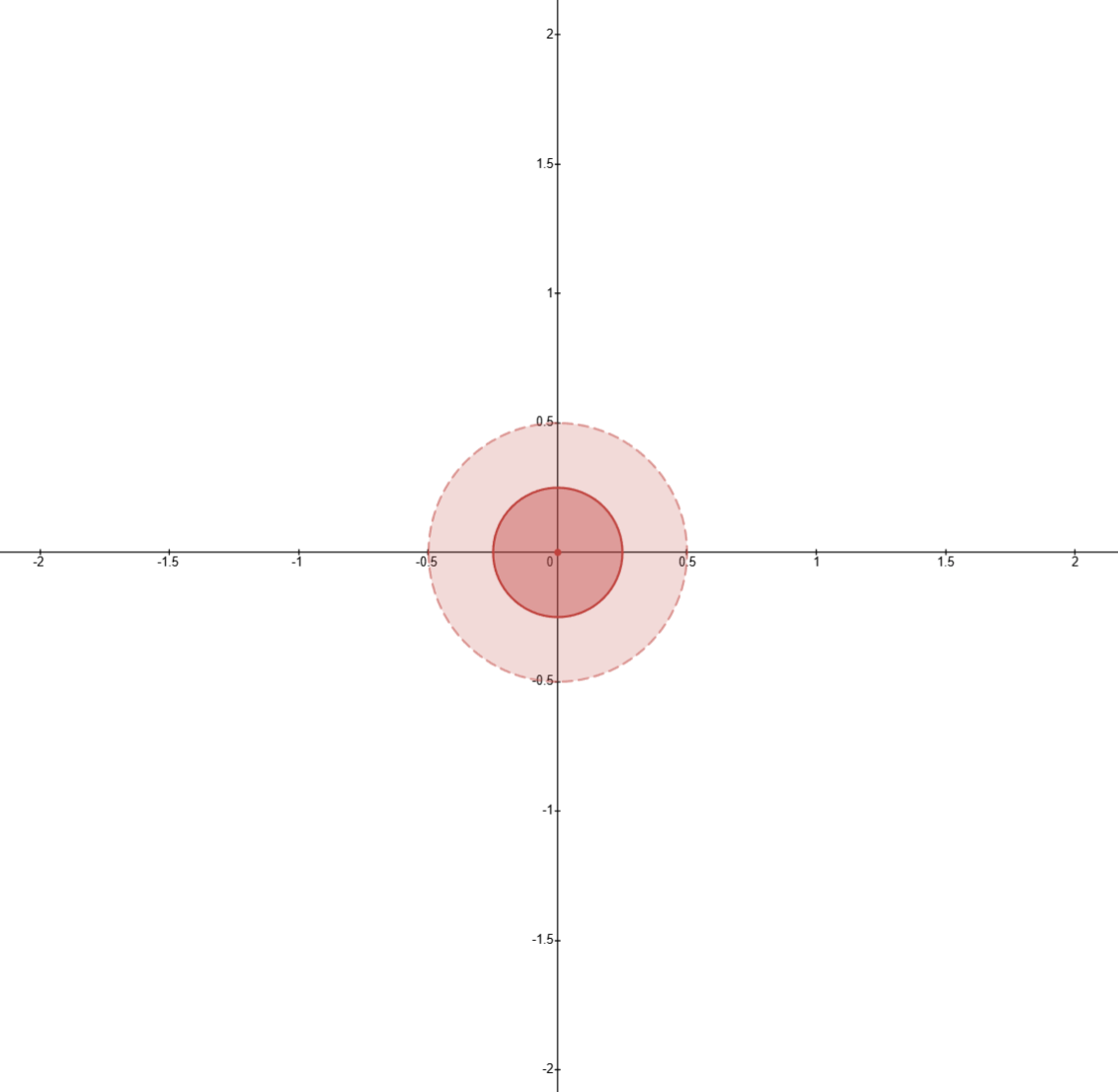}
    \caption{The blob at time \(t=0.25\).}
    \label{fig:the-blob-at-time}
\end{figure}

At time \(t=\delta\), we then begin drawing out in both directions (clockwise and counterclockwise) from the rightmost point of the blob at an angle from the origin of \(\arccos(2/\lambda)\). This angle gives a distance of \(\lambda/2\) per second on the hypotenuse versus one unit per second on the radial line, which is thus possible as the blob therefore touches the wall at a rate of \(\lambda/2\) units per second, which is permissible. This means that we draw along the outside of the blob at all times, and that we may do so indefinitely. 

These walls form two logarithmic spirals with angle coefficient \(k=\tan(\pi/2-\arccos(2/\lambda))\). 

\begin{figure}[H]
    \centering
    \includegraphics[trim=5cm 5cm 5cm 5cm, clip=true, width=0.6\linewidth]{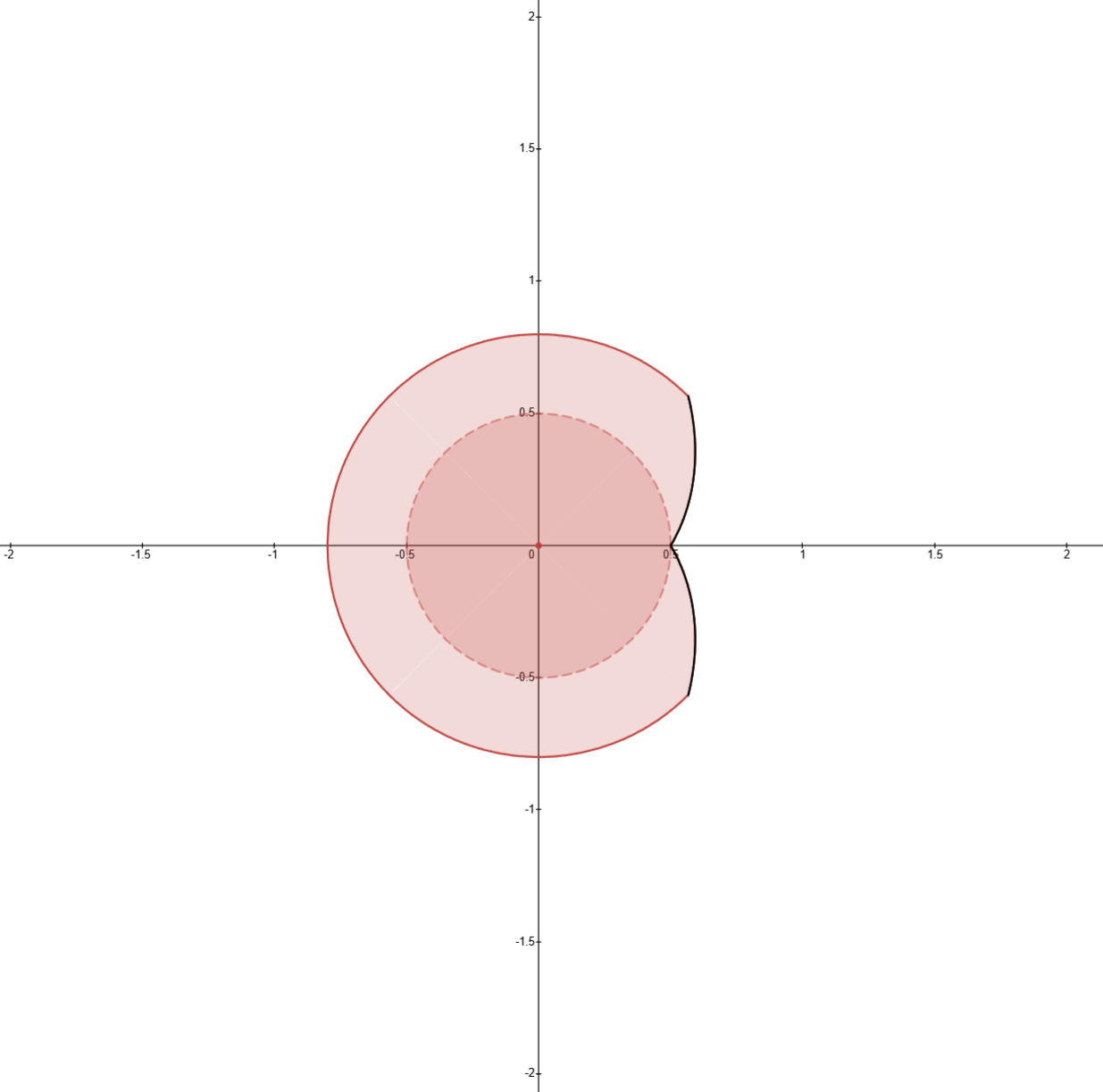}
    \caption{The blob and the walls which touch it partway through containment. }
    \label{fig:the-blob-and-the-walls-which}
\end{figure}

This then continues until we contain the blob on the negative \(x\)-axis. 

\begin{figure}[H]
    \centering
    \includegraphics[width=0.75\linewidth]{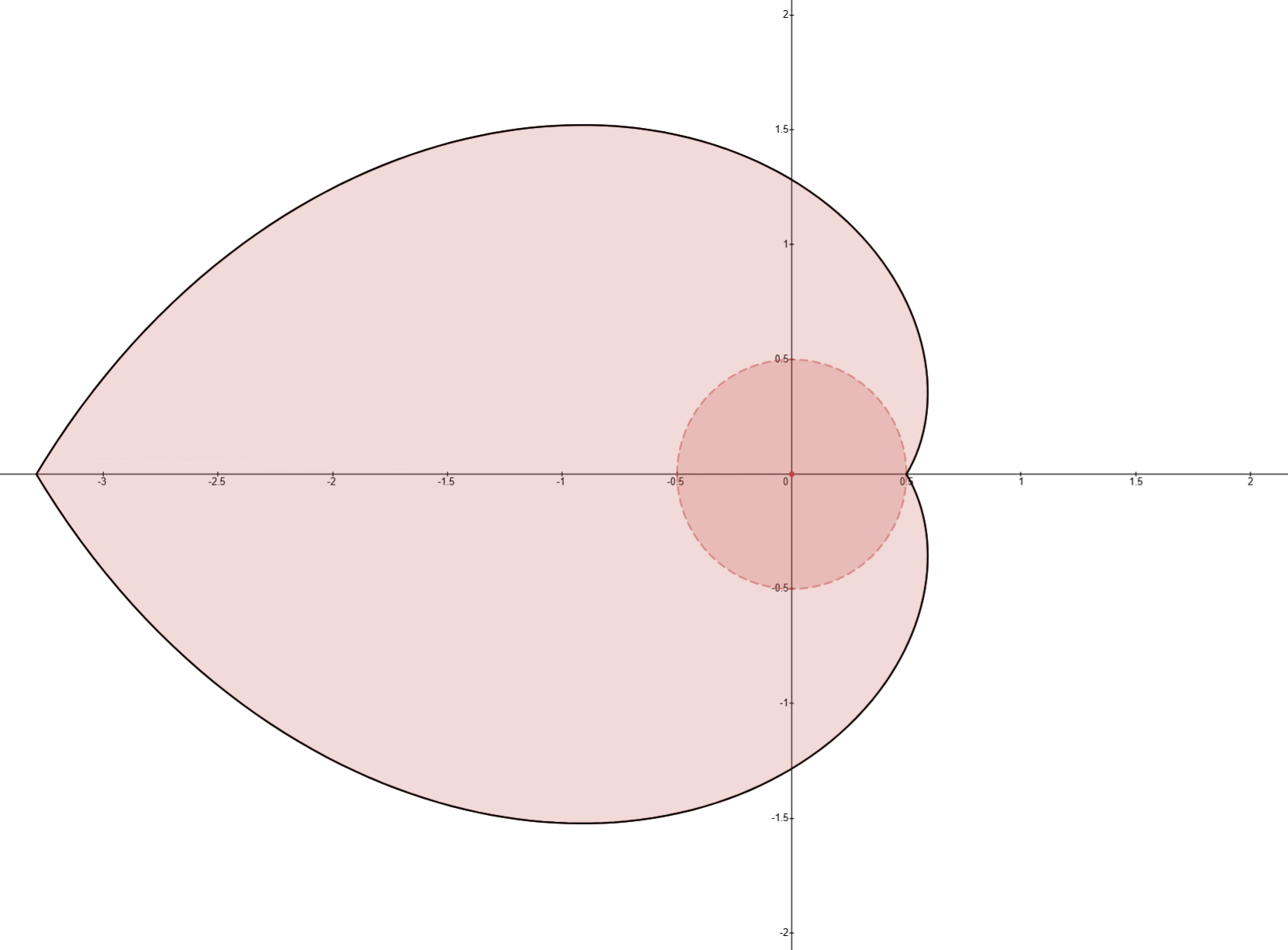}
    \caption{The blob contained by the ``teardrop" method.}
    \label{fig:the-blob-contained-by-the-teardrop}
\end{figure}

As \(\lambda\to 2\), the teardrop shape becomes more and more linear. 

\begin{figure}[H]
    \centering
    \includegraphics[width=0.75\linewidth]{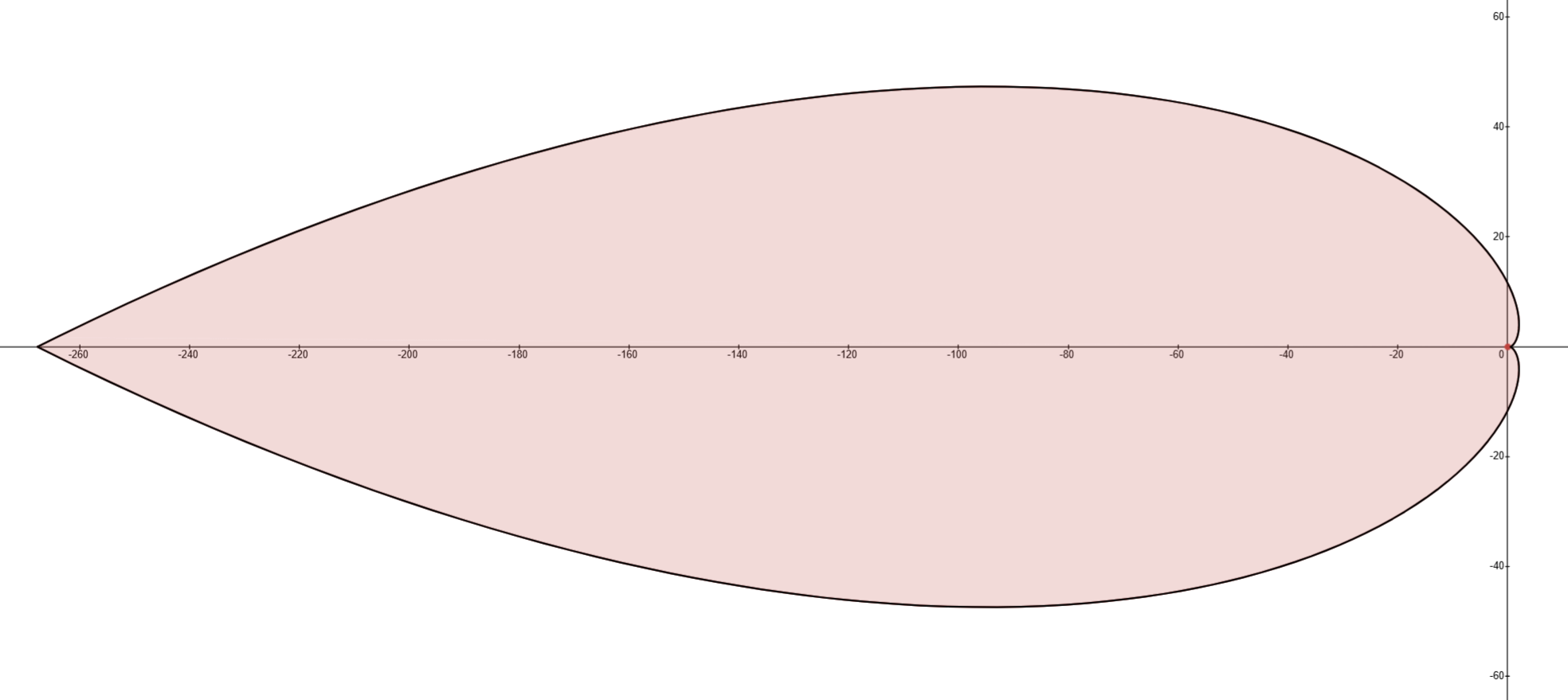}
    \caption{The blob contained by the ``teardrop" method with a \(\lambda\) fairly close to 2.}
    \label{fig:the-blob-contained-by-the-teardrop-2}
\end{figure}

And for \(\lambda\gg2\), the shape becomes nearly circular. 

\begin{figure}[H]
    \centering
    \includegraphics[width=0.6\linewidth]{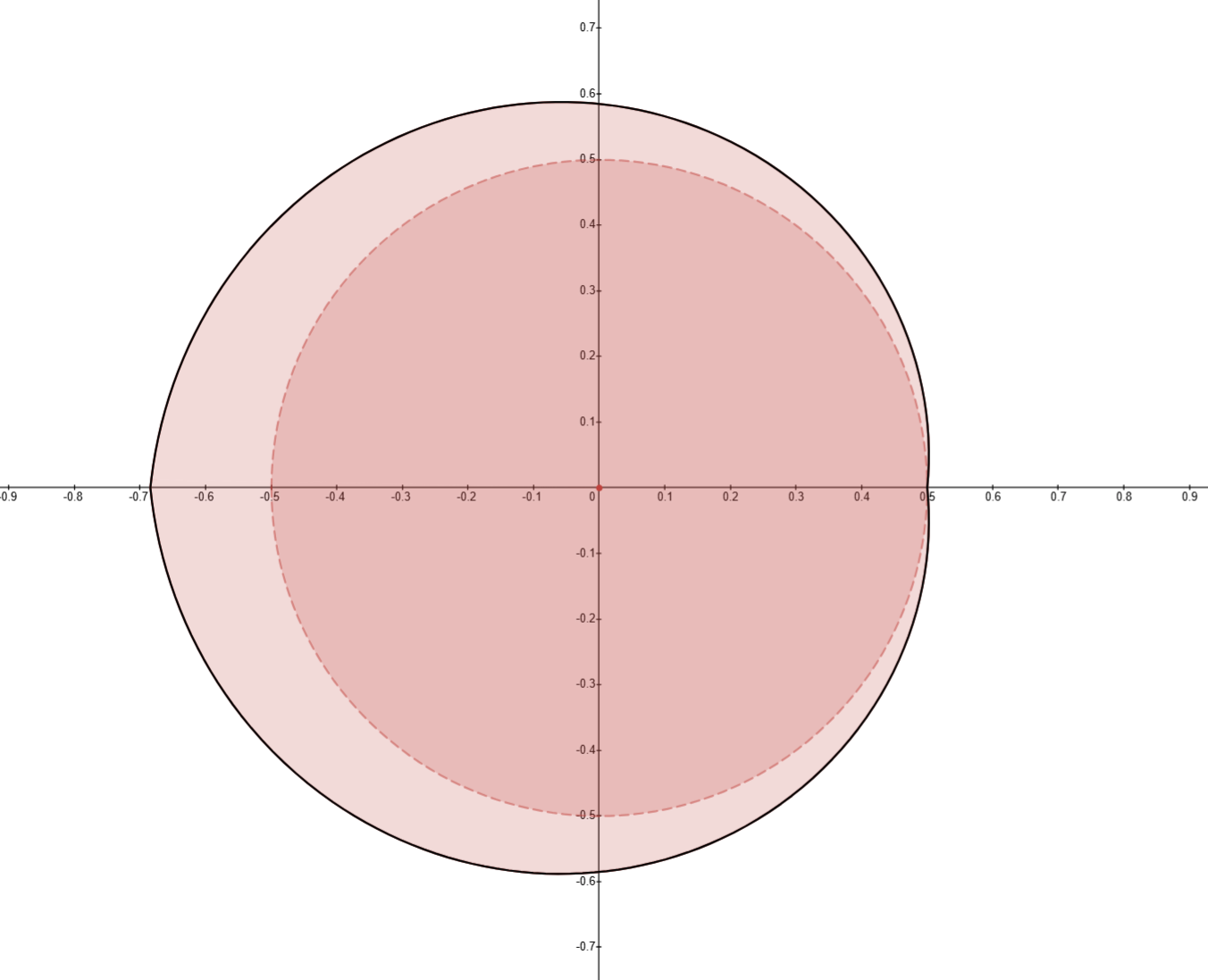}
    \caption{The blob contained by the ``teardrop" method with a \(\lambda\) much larger than 2.}
    \label{fig:the-blob-contained-by-the-teardrop-3}
\end{figure}

This then places the best known upper bound for the critical \(\lambda\) as 2.
\end{proof}

However, this is not the best known strategy for \(\lambda\geq 2\). The best known strategy does not sit idle until time one, and instead begins building a vertical wall at the furthest right point of the exclusion radius immediately.

\begin{figure}[H]
    \centering
    \includegraphics[width=0.6\linewidth]{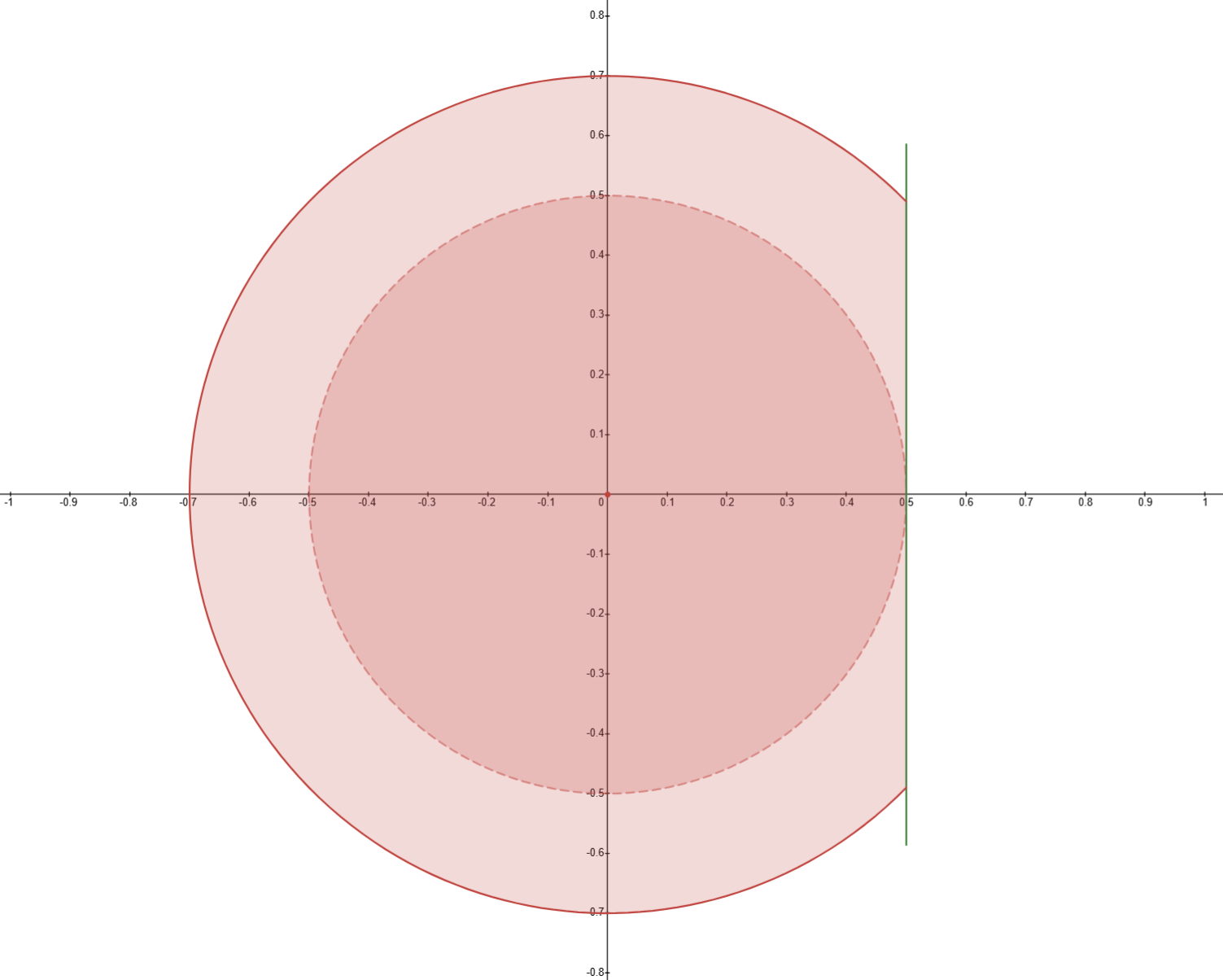}
    \caption{The blob during the construction of this first wall.}
    \label{fig:the-blob-during-the-construction-of}
\end{figure}

Once the angle from the wall to the center becomes \(\arccos(2/\lambda)\), we begin drawing our two spirals as before. 

\begin{figure}[H]
    \centering
    \includegraphics[width=0.75\linewidth]{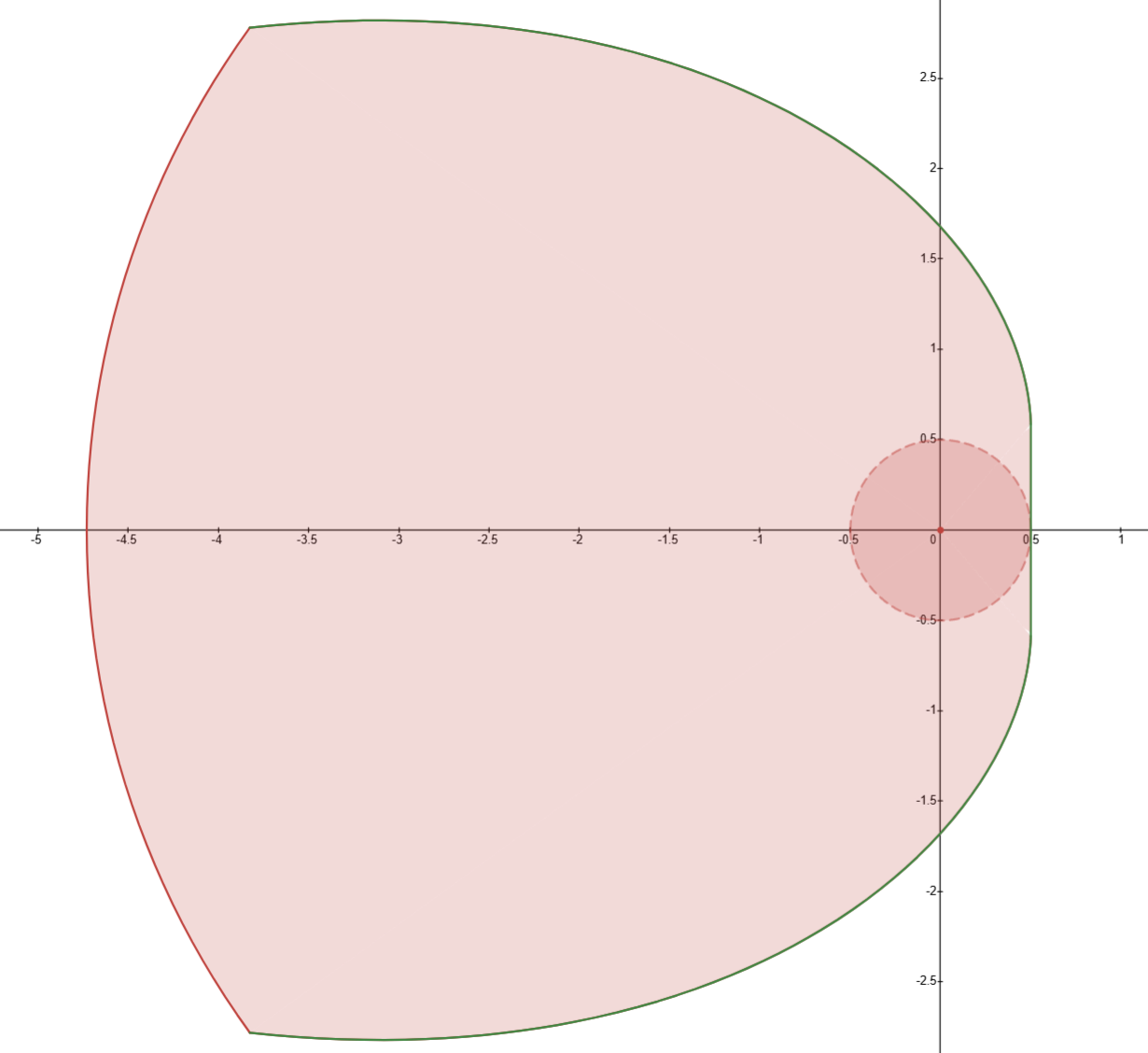}
    \caption{The blob during the construction of the spirals.}
    \label{fig:the-blob-during-the-construction-of-2}
\end{figure}

Containment happens significantly faster and with much less required area than with the previous method. 

\begin{figure}[H]
    \centering
    \includegraphics[width=0.75\linewidth]{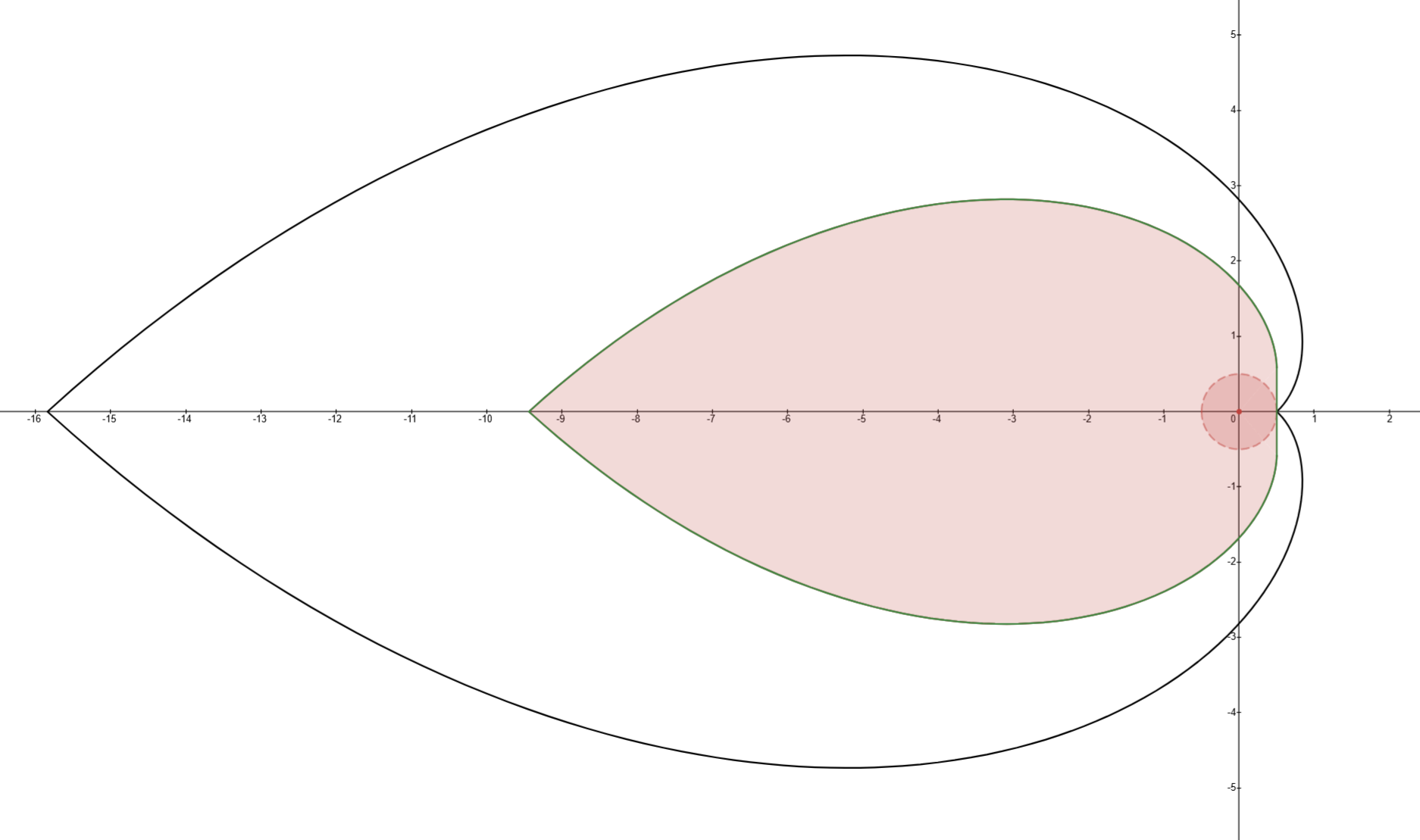}
    \caption{The blob at containment with the new strategy (in green). The old strategy is shown in black for the same \(\lambda\).}
    \label{fig:the-blob-at-containment-with-the}
\end{figure}

While a slightly stronger strategy does exist in this slightly weakened case of the blob, this is the best-known strategy that works on the original problem. With the Blob starting at radius 1, the vertical wall must be drawn at a distance of \(x=\frac{\cos(\theta)\lambda}{\lambda-2\tan(\theta)\cos(\theta)}\), where \(\theta\) is the angle where the spiral goes perfectly vertical. This means that this is the closest distance where we could draw a vertical wall that would reach the angle where the logarithmic spiral goes vertical, meaning we can then create the log spiral structure which contains the blob.

No further improvement has been made on the upper bound for \(\lambda_C\) in the general problem. (For barriers restricted to a single self-closing logarithmic spiral, as opposed to the more general two-armed construction above, the sharp threshold is known to be \(\lambda\approx2.6144\)~\cite{kleinlangetepelevcopoulos2015,bianchinizizza2025}. We do not use that restricted result here, though we did independently arrive at the same threshold.)

\section{The Vine}
\label{sec:vine}

We hereby define the problem known as the \(n\)-vine. 

The \(n\)-vine, similar to the blob, begins as a single point at the origin of 2D space, and has a small exclusion radius of arbitrary positive length around it. 

\begin{figure}[H]
    \centering
    \includegraphics[width=0.75\linewidth]{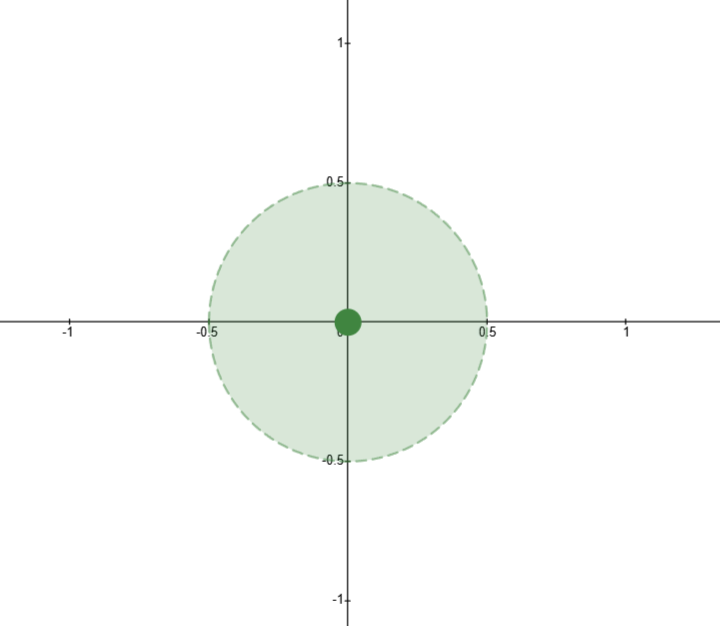}
    \caption{The vine at time zero.}
    \label{fig:the-vine-at-time-zero}
\end{figure}

It may, at any given time, be propagating up to \(n\) curves from any point of the vine, which grow at a rate of 1 unit per second. These curves are referred to as tendrils. Tendrils may be started and stopped at the vine's will. 

\begin{figure}[H]
    \centering
    \includegraphics[width=0.75\linewidth]{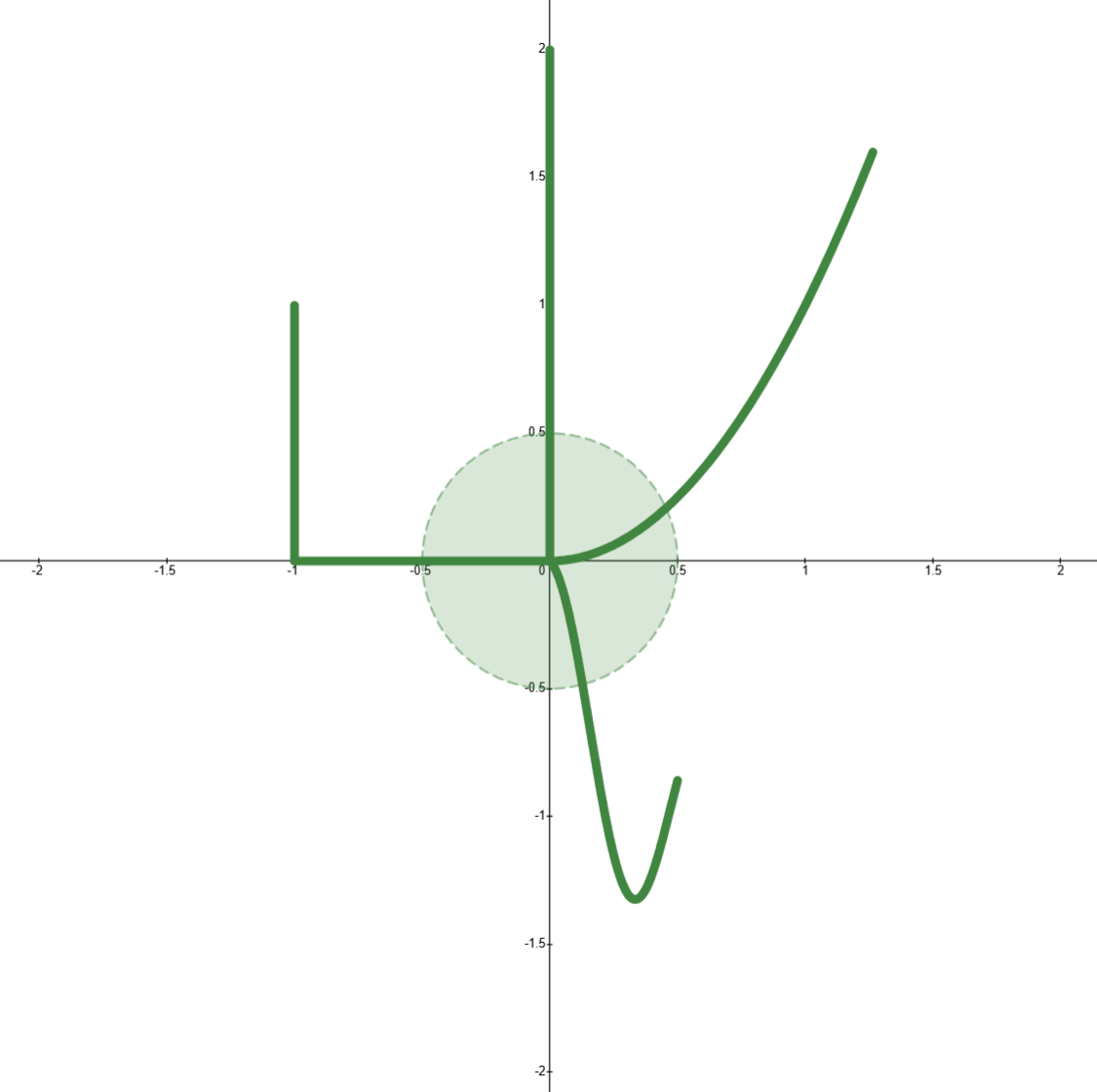}
    \caption{A potential 4-vine at time 2.}
    \label{fig:a-potential-4vine-at-time-2}
\end{figure}

At the same time, we may build walls at a rate of \(\lambda\) units per second to contain the vine, which prohibit the motion of tendrils. However, if the vine touches a wall at any point, the vine will creep along the wall in both directions at a rate of 1 unit per second, eventually enveloping it if left alone. This does not occupy or consume the tendril, so the \(n\)-vine can be expanding at a rate of 1 unit per second with up to \(n\) tendrils and at an arbitrarily high number of points on walls. The points on walls along with the vine creeps we refer to as ``fronts". 

\subsection{Rigorous Definition of the Vine}
\label{sec:vinedef}

First, we redefine a wall as not only a rectifiable, compact curve, but also as an incision in the topological surface that is \(\R^2\). We also define the set of walls at a time \(t\) as a weakly monotonic function \(W(t)\). We then take the topological space given by the metric space \(\R^2\setminus W(t)\) composed with the shortest-path metric over the set of paths in \(\R^2\setminus W(t)\), and we define the \textit{court} \(C(t)\) to be its closure. 

\begin{figure}[H]
    \centering
    \includegraphics[width=0.75\linewidth]{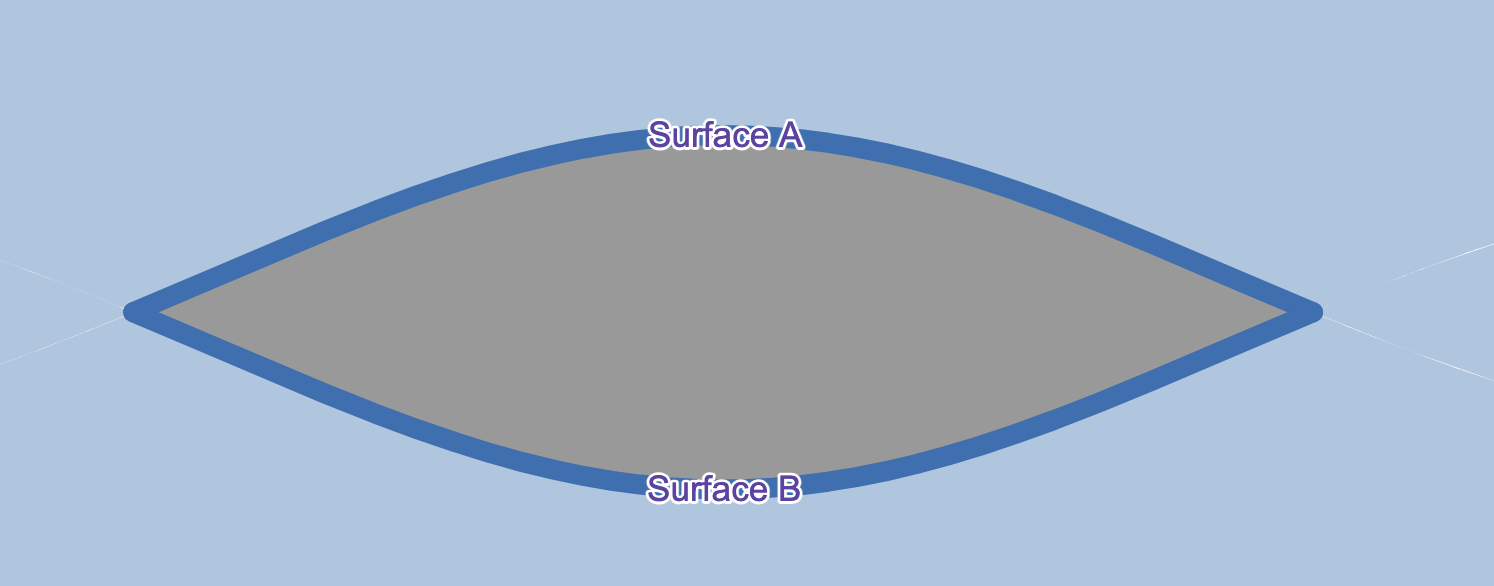}
    \caption{An incision in \(\R^2\) from a wall, and the resulting surfaces from the closure. The incision has been pulled open for clarity in this figure, but in the future we will continue to draw walls as curves and show surfaces only when the vine has crept along that wall.}
    \label{fig:a-cut}
\end{figure}

We define the two \textit{surfaces} of a wall as the two rectifiable curves in \(C(t)\) that are the boundary points created by this incision, which therefore means that the set of surfaces is the set of boundary points of \(\R^2\setminus W(t)\). Note that surfaces are identical in \(\R^2\) to the wall which produced them. Surfaces are adjacent to at most two other surfaces, as the endpoints of a wall will be contained in at most two different surfaces and each surface contains at most two endpoints. Finally, we define the set of surfaces at a time \(t\) as a weakly monotonic function \(S(t)\). This therefore means that, as a set, \(C(t)=(\R^2\setminus W(t))\cup S(t)\).

Now, we define the vine as the set of points on \(C(t)\) given by a function \(V(P, t)\) of time \(t\) and a \textit{plan} \(P\). \(V(P,t)\) is also weakly monotonic, and is constructed through two properties: 

First, the creeping property. If, at a time \(t\), \(V(P,t)\) contains a point \(p\) on a surface, then, for all \(x\geq 0\), \(V(P,t+x)\) contains all points on that surface within length \(x\) of \(p\). 

Second, the tendril property. For the \(n\)-vine, \(P(k,S_C(t),t)=f_k(t)\), where the integer \(1\leq k\leq n\) is the \(k\)th tendril and \(S_C(t)\) is the set of subsections of surfaces in the vine at time \(t\), giving \(f_k(t)\), the curve of the \(k\)th tendril. \(f_k(t)\subseteq V(P,t)\) for all \(t\), \(f_k(t)\) is a curve on \(C(t)\), \(df_k(t)/dt\in [0,1]\cup \{\infty\}\) for all \(t\), \(f_k\) is right-continuous, and finally, for all \(t\) such that \(df_k(t)/dt=\infty\), \(f_k(t)\in V(P'_{\epsilon,k},t)\) for some \(\epsilon>0\), where \(P'_{\epsilon,k}(j,S_C(u),u)=f_j(u)\) for all \(0\leq u\leq t\) for \(j\neq k\), but \(P'_{\epsilon,k}(k,S_C(u),u)=f_k(\min(u,t-\epsilon))\) for all \(0\leq u\leq t\). 

The containment of the vine is successful if the vine is contained within a closed figure at time \(T\) and if, at all times \(t\), \(L(S_C(t))\leq \lambda t\), where \(L(S)\) for a set of subsections of surfaces \(S\) is the length of wall required for those subsections of surfaces, summing the lengths in \(\R^2\) where, if both surfaces for a wall are in \(S\), they are identical in \(\R^2\) and thus are only counted once in \(L(S)\).

\begin{figure}[H]
    \centering
    \includegraphics[width=1.0\linewidth]{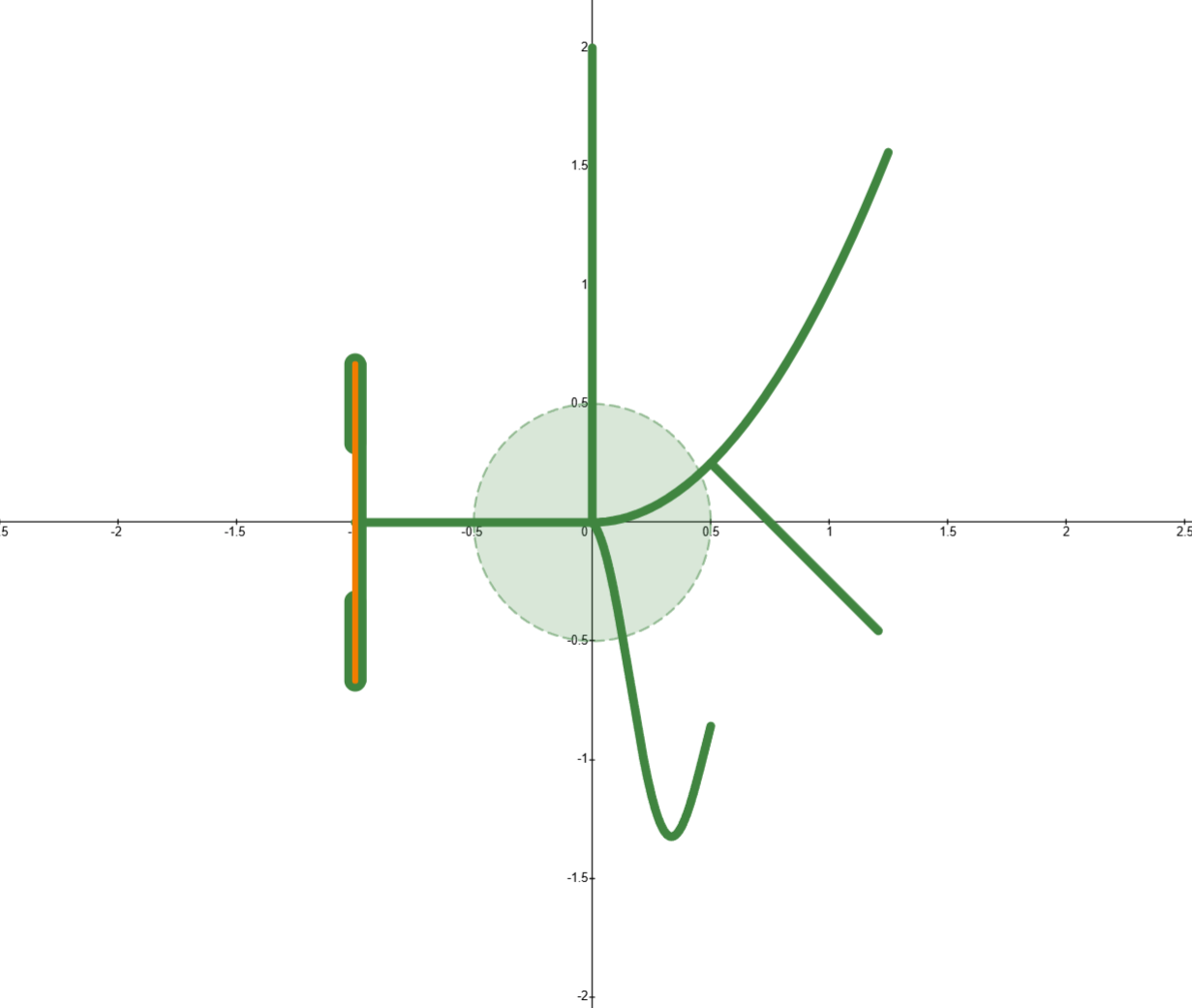}
    \caption{A possible state of the 4-vine at time 2. At this time, all four tendrils are currently expanding, and the vine is creeping along the wall to the left at two points. The tendril which touched the wall was halted, and another tendril was started to replace it as a branch on the eastward tendril.}
    \label{fig:a-possible-state-of-the-4vine}
\end{figure}

We note that, as the function which produces the tendrils, the plan \(P\), takes as parameters only the tendril \(k\), the time \(t\), and the set of surfaces contacted \(S_C(t)\) at time \(t\), the vine is only able to act in accordance with the walls that it has make contact with. 

The vine's motion is thus unchanged by any walls which it has not yet made contact with, and therefore the vine cannot act in a reactive fashion and must instead be fully deterministic based on solely what it has ``seen."

\begin{figure}[H]
    \centering
    \includegraphics[width=0.9\linewidth]{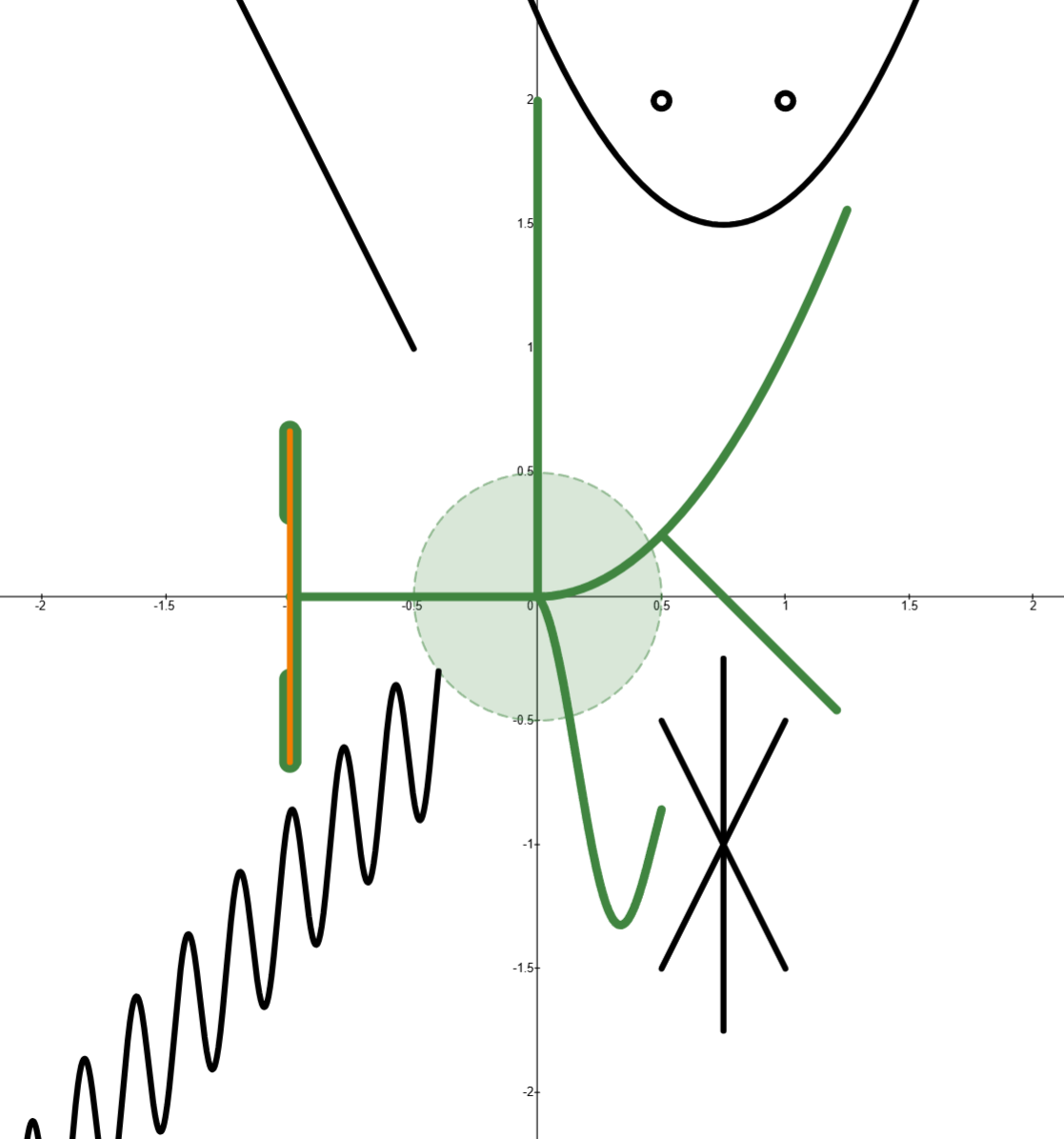}
    \caption{The above 4-vine at time 2, with untouched walls in black. These walls do not affect the vine in any way, as the vine cannot ``see" them. The only wall which the vine can ``see," the wall in orange, is the only wall which may affect the vine's movements until the vine touches another wall.}
    \label{fig:the-above-4vine-at-time-2-2}
\end{figure}

We now introduce a series of lemmas.

\begin{lemma}[Weakening of the Blob]
\label{lma:blob-weakening}
Any solution that contains the blob for some \(\lambda\) must be able to contain the \(n\)-vine for any finite \(n\). 
\end{lemma}
\begin{proof}[Proof]
As the vine moves exclusively by creeping at a rate of 1 and moving tendrils at a rate of at most 1, each point of the \(n\)-vine at every time \(t\) is within a distance of \(t\) of the origin, accounting for the walls which existed at the time of the path of length \(t\). Therefore, the \(n\)-vine at time \(t\) is always contained within the blob at time \(t\). Therefore, if there is some strategy which contains the blob for some \(\lambda\), the vine cannot touch more wall than the blob can, and thus the strategy also contains the vine for that same \(\lambda\).
\end{proof}

\begin{figure}[H]
    \centering
    \includegraphics[width=1.0\linewidth]{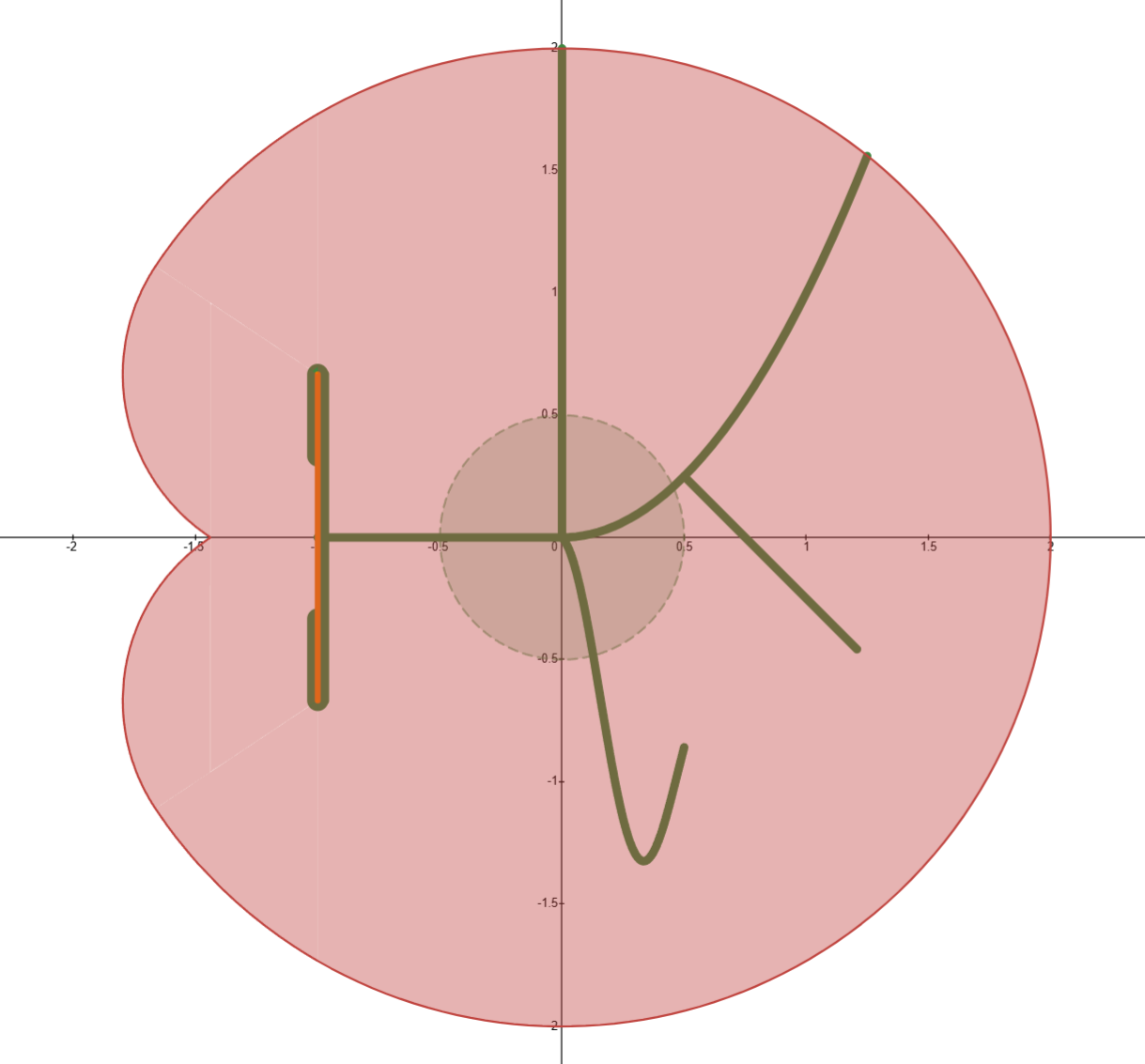}
    \caption{The above 4-vine at time 2, compared to the blob at time 2. It is apparent that the vine is contained within the blob.}
    \label{fig:the-above-4vine-at-time-2}
\end{figure}

\begin{lemma}[Front Erasure]
\label{lma:front-erasure}
Each front can only cease to exist when it runs into another front.
\end{lemma}
\begin{proof}[Proof]
A front moving in a direction exists at a point on a surface where, immediately before that point along the surface in that direction, there is no vine. The front will continue to exist until the surface in front of it is touched by the vine. If this is to occur, then there must be a point of the surface touched by the vine immediately before which (in the opposite direction to the original front) there is no vine along the surface. This, therefore, must be a front moving in the opposite direction to the original front. Therefore, if any front ever encounters a point along the surface which is touched by the vine, that point must be another front. 
\end{proof}

We consider the cardinal \(n\)-vine to be the \(n\)-vine with the following strategy: picking \(n\) equally spaced directions, the vine sends out its tendrils along rays in those directions. If the furthest point of the tendril along that direction is blocked from going further by a wall, the vine stops use of that tendril until it can move along that ray again. 

\begin{figure}[H]
    \centering
    \includegraphics[width=0.9\linewidth]{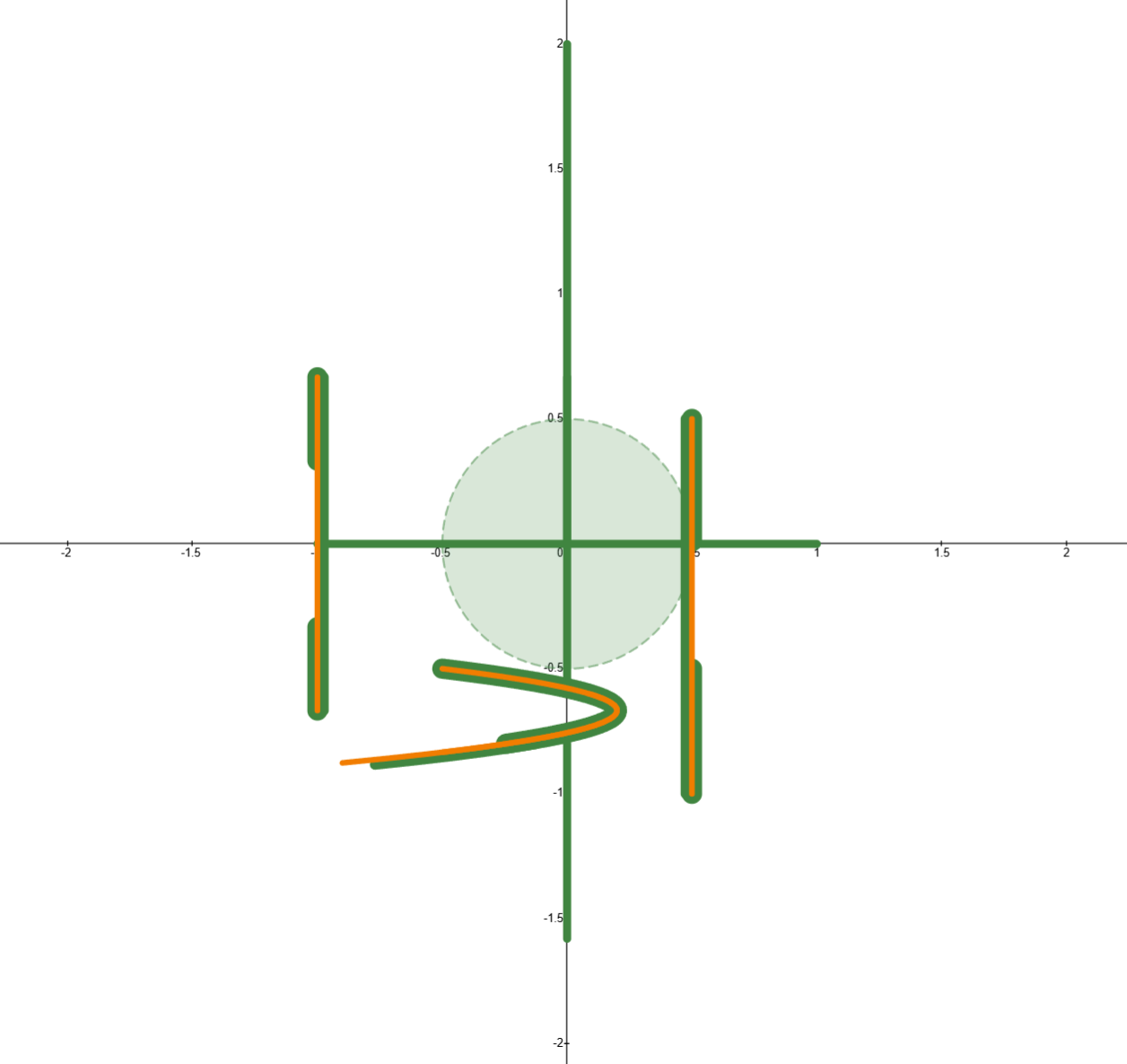}
    \caption{The cardinal 4-vine interacting with some walls at time 2. The north tendril is unobstructed, the east tendril was halted until the vine crept around to the other side of its wall, the south tendril was halted until the vine crept along the curve to the other side, and the west tendril is still halted as the vine has not crept back around to the \(x\) axis yet. Note that the south tendril never moved along the section of the \(y\) axis inside the curve, as by the time the vine crept to the first interior point on the \(y\) axis, it had already crept to the exterior point on the \(y\) axis below it, and thus the tendril began from there.}
    \label{fig:the-cardinal-4vine-interacting-with-some}
\end{figure}

\begin{lemma}[Dimensions of the Cardinal Vine]
\label{lma:dim-cardinal-vine}
Given a cardinal \(n\)-vine with \(n\) equiangular rays from the origin, the distance \(D_r(t)\) from the origin to the furthest point of the vine along a given ray \(r\) at time \(t\) is at least the total free time \(T_{F,r}(t)\) that the tendril has spent not-halted and moving along r at time \(t\). In other words, \(T_{F,r}(t)\leq D_r(t)\) for all \(t\geq 0\). 
\end{lemma}
\begin{proof}[Proof]
If the tendril for ray \(r\) is moving, then it must be moving at a rate of 1 unit per second away from the origin along ray \(r\). Therefore, for every \(x\) seconds that the tendril has spent not-halted and moving along \(r\), it has moved a distance of \(x\) away from the origin. And, since the distance of the furthest point of the vine along a ray is a weakly increasing function, and the distance has increased by at least \(T_{F,r}(t)\) at time \(T\), we see that \(T_{F,r}(t)\leq D_r(t)\) for all \(t\geq 0\). 
\end{proof}

\begin{lemma}[Wall Requirements for Containment]
\label{lma:wall-reqs-contain}
Given a state of the vine at time \(t\), and given that the vine will be contained at some time \(T\), the total free surface \(S_F\), surface which has not been touched by the vine, that will exist at time \(T\) is no less than the perimeter \(H(t)\) of the convex hull around the vine at time \(t\). In other words, \(S_F\geq H(t)\) for all \(t\geq 0\). 
\end{lemma}
\begin{proof}[Proof]
If the vine is to be contained at time \(T\), then all points of the vine must be contained inside a closed figure. The outside of the figure must, therefore, be both untouched by the vine and also form a closed figure which contains all points of the vine. This closed figure has length at least the length of the convex hull around the points of the vine at time \(T\), so the closed figure has wall length at least \(H(T)\). Moreover, as the vine is ever-expanding, we have that \(H(T)\geq H(t)\) for all \(0\leq t\leq T\). Therefore, the closed figure has wall length at least \(H(t)\). 

Therefore, as at least one side of the wall must be untouched by the vine, as there must be no vine on the outside of the figure, we see that the total length of free surface \(S_F\) must be at least \(H(T)\) and thus at least \(H(t)\). 
\end{proof}

\begin{lemma}[Wall Requirements for Containment on Ray Motion]
\label{lma:wall-reqs-ray-motion}
Given a state of the cardinal \(n\)-vine at time \(t\), with rays \(r_1,r_2,\dots,r_n\) and furthest points \(p_{r_1}(t)\) along each ray at time \(t\), and given that the vine will be contained at some time \(T\), the total free surface \(S_F(T)\) is no less than the perimeter \(H(\{p_{r_1}(t),p_{r_2}(t),\dots,p_{r_n}(t), o\})\) of the convex hull  around those furthest points at time \(t\) and the origin. In other words, \(S_F(T)\geq H(\{p_{r_1}(t),p_{r_2}(t),\dots,p_{r_n}(t), o\})\) for all \(t\geq 0\). 
\end{lemma}
\begin{proof}[Proof]
Those \(n+1\) points are points that will be contained inside of the vine at time \(T\). Therefore,  \(H(t)\geq H(\{p_{r_1}(t),p_{r_2}(t),\dots,p_{r_n}(t), o\})\) for all \(t\geq 0\). Therefore, by Lemma~\ref{lma:wall-reqs-contain},  \(S_F(T)\geq H(\{p_{r_1}(t),p_{r_2}(t),\dots,p_{r_n}(t), o\})\) for all \(t\geq 0\). 
\end{proof}

\begin{lemma}[Creeping Exclusivity]
\label{lma:creeping-exclusive}
A surface can only be crept on once. 
\end{lemma}
\begin{proof}[Proof]
Once a front creeps over a surface, the vine touches the surface. Therefore, the vine will always touch that surface. Therefore, the vine will never again not be touching that surface. Therefore, that surface cannot be crept on again by any other front. 
\end{proof}

\begin{lemma}[Creeping Wall Requirement]
\label{lma:creeping-wall}
If a front creeps for \(x\) seconds, it touches a length of \(x\) surface. 
\end{lemma}
\begin{proof}[Proof]
A front creeps at 1 unit per second along surfaces. Therefore, after \(x\) seconds, it has crept along a length of \(x\) surface.  
\end{proof}

\begin{lemma}[Total Creeping Wall Requirement]
\label{lma:total-creeping}
The sum across all fronts of the time each front exists is less than or equal to the length of surface the vine touches. 
\end{lemma}
\begin{proof}[Proof]
By Lemma~\ref{lma:creeping-wall}, if a front exists for \(x\) seconds it touches a surface length of \(x\). By Lemma~\ref{lma:creeping-exclusive}, fronts cannot touch the same surface twice. Therefore, each front across \(x\) seconds of existence touches a length of \(x\) unique surface, and thus the total sum across all front existences cannot exceed the amount of surface touched. 
\end{proof}

\begin{lemma}[Wall Surface Continuity]
\label{lma:wall-surface-cont}
If the vine touches a surface at some time \(t\), if the vine is allowed to creep forever with no more walls ever added, the vine will touch the opposing surface of the wall if and only if the wall whose surface the vine is touching is not part of a loop whose other face contains no point of the vine. 
\end{lemma}
\begin{proof}[Proof]
If the wall is not part of a loop, then the bounded region on the other side of the wall, on the untouched surface, is the same bounded region the vine is currently touching the wall on, on the touched surface. Therefore, the vine may creep around the entirety of the wall, the two surfaces of that wall must be connected as well. Therefore, after time time, the vine will creep around to that other surface.

If the wall is part of a loop, however, and in the face contained in that loop there is no vine, then as the vine may not cross the wall, there is no way for the vine to get to the other face across the loop. Therefore, the vine will never creep around to that other surface.
\end{proof}

\begin{lemma}[Loop Creeping]
\label{lma:loop-creeping}
If the vine touches a loop of wall, where the face on the other side of the loop there is no vine, the vine is either already successfully contained or will, if the vine is allowed to creep forever with no more walls ever added, creep around the loop and contain the entire loop strictly inside itself.
\end{lemma}
\begin{proof}[Proof]
As all points of the vine lie on one side of the loop, we may consider the vine to be ``inside" the loop if the vine is contained in the region with finite area and ``outside" the loop otherwise. If the vine is inside the loop, then the vine is clearly successfully contained as it is restricted to a bounded figure with finite area. If the vine is outside the loop, then as the loop is formed of connected surfaces, if the vine is not successfully contained, it must be able to creep entirely around it as the surfaces outside the loop are all connected. This would mean the vine will eventually surround the loop externally, bordering it on the outside, containing all outside surfaces, and containing the entire loop strictly inside itself.
\end{proof}

\begin{lemma}[Creeping on Rays]
\label{lma:creeping-rays}
If, in the cardinal \(n\)-vine, a tendril moving along its ray touches a piece of wall and is halted, if the vine is allowed to creep forever with no more walls ever added, the vine is contained or the tendril will eventually un-halt and be able to continue along its ray forever. 
\end{lemma}
\begin{proof}[Proof]
If the wall that the tendril touches is part of a loop, then by Lemma~\ref{lma:loop-creeping}, the tendril will reach the other side of that loop on the outside (and thus be able to continue on along its ray) if the vine is not contained. 

If the wall that the tendril touches is not part of a loop, then by Lemma~\ref{lma:wall-surface-cont}, the tendril will reach the other surface of that wall (and thus be able to continue on along its ray). 

Therefore, if the vine is not contained, the tendril will thus be able to continue along its ray.

Should the vine encounter another piece of wall along its ray, by the above, it will only not be able to continue along its ray if the vine is contained. As such, if the vine is not contained, the tendril will thus be able to continue along its ray forever, as it can only encounter finitely many such pieces of wall. 
\end{proof}

\begin{lemma}[Wall Requirements for Creeping on Rays]
\label{lma:wall-reqs-rays}
If, in the cardinal \(n\)-vine, a tendril moving along its ray touches a piece of wall and is halted at a time \(t\), and the tendril is un-halted at time \(t+x\) (and not before) to continue along its ray, then over the course of that time \(x\) there must have been at least two fronts creeping, each covering a distance of \(x\). 
\end{lemma}
\begin{proof}[Proof]
This is trivially true if \(x=0\). So assume \(x>0\). 

By Lemma~\ref{lma:creeping-wall}, we must only show that there are at least two fronts creeping between times \(t\) and \(t+x\). We may say that the wall along the ray which is reached by the vine at time \(t+x\) is, at time \(t\), untouched (if it is not yet built, then of course it must be untouched). Therefore, as the vine reaches points either by sending a tendril along that ray or by creeping, and the tendril along this ray is halted, the vine must reach this point by creeping. 

Therefore, at time \(t+x\), creeping must occur and thus there must be at least one front. 

By Lemma~\ref{lma:front-erasure}, fronts can only cease to exist when they hit another front, and as fronts are created in pairs when a tendril touches an untouched surface of wall or a new wall is drawn to touch a touched surface of wall, there must be an even number of fronts at any given time. 

Therefore, it suffices to show that, at no time between \(t\) and \(t+x\), were there zero fronts. 

There are zero fronts at a given time \(t\) if and only if there are no sections of surfaces touched by the vine adjacent to sections of surfaces untouched by the vine. In other words, every surface that is touched by the vine is only adjacent to surfaces that are also touched by the vine. As such, the whole connected curve of every surface which has any point of the vine must be entirely contained within the vine. 

By Lemma~\ref{lma:wall-surface-cont} and Lemma~\ref{lma:loop-creeping}, this would mean that either the vine has crept over to the other side of the wall that the tendril touched and was halted for, meaning the tendril is continuing along its ray, or the vine-touched side is contained within a loop. 

But then, either there is no vine outside of the loop, in which case it cannot possibly creep to the outside of the loop, or there is vine outside of the loop. If there is vine outside of the loop, then, as there is vine inside of the loop, there must be vine on both sides of the loop. As the vine is contiguous, there must be vine touching both the inside and the outside surfaces of the loop. Therefore, the entire loop must be touched by the vine inside and outside, meaning that the other side of the vine-touched side that halted the tendril is, in fact, vine-touched and thus the tendril is continuing along its ray already. 

Therefore, if there is no front active at any time \(t+z\) strictly between \(t\) and \(t+x\), then at that time the tendril will be un-halted and will be continuing along its ray at time \(t+z\). This is a contradiction. Therefore, at all times between \(t\) and \(t+x\), there are at least two fronts active. 
\end{proof}

\begin{lemma}[Other Wall Requirements for Creeping on Rays]
\label{lma:other-wall-reqs-rays}
If, in the cardinal \(n\)-vine, a tendril moving along its ray touches a piece of wall and is halted at a time \(t\), the vine is first successfully contained at time \(t+x\), and the tendril is not un-halted before time \(t+x\), then over the course of that time \(x\) there must have been at least two fronts creeping, each covering a distance of \(x\). 
\end{lemma}
\begin{proof}[Proof]
If the vine is first successfully contained at time \(t+x\), then it is not successfully contained prior to \(t+x\). Therefore, by the above argument, if the surface that the tendril touched is inside of a loop, there is vine outside of that loop as otherwise the vine would be successfully contained. Therefore, there is some untouched surface contiguous to the touched surface, and therefore there are at least two fronts.

\end{proof}

Now we begin our analysis of the cardinal \(1\)-vine. 

\begin{theorem}
\label{thm:cardinal-1-vine}
The cardinal \(1\)-vine has critical \(\lambda\) equal to \(1\): for \(\lambda>1\) it can be contained, and for \(\lambda\leq 1\) it cannot be. Consequently, since the vine is a weakening of the blob, the blob cannot be contained for any \(\lambda\leq1\) (Bressan~\cite{bressan2007}).
\end{theorem}

\begin{proof}[Proof of sufficiency]
Assume \(\lambda=1+\epsilon\) for \(\epsilon>0\). Assume this tendril moves directly east. We begin by constructing a 1-high wall one unit to the east. The vine sends out its tendril to the east, hits the wall at time 1, and creeps along it. 

\begin{figure}[H]
    \centering
    \includegraphics[width=0.75\linewidth]{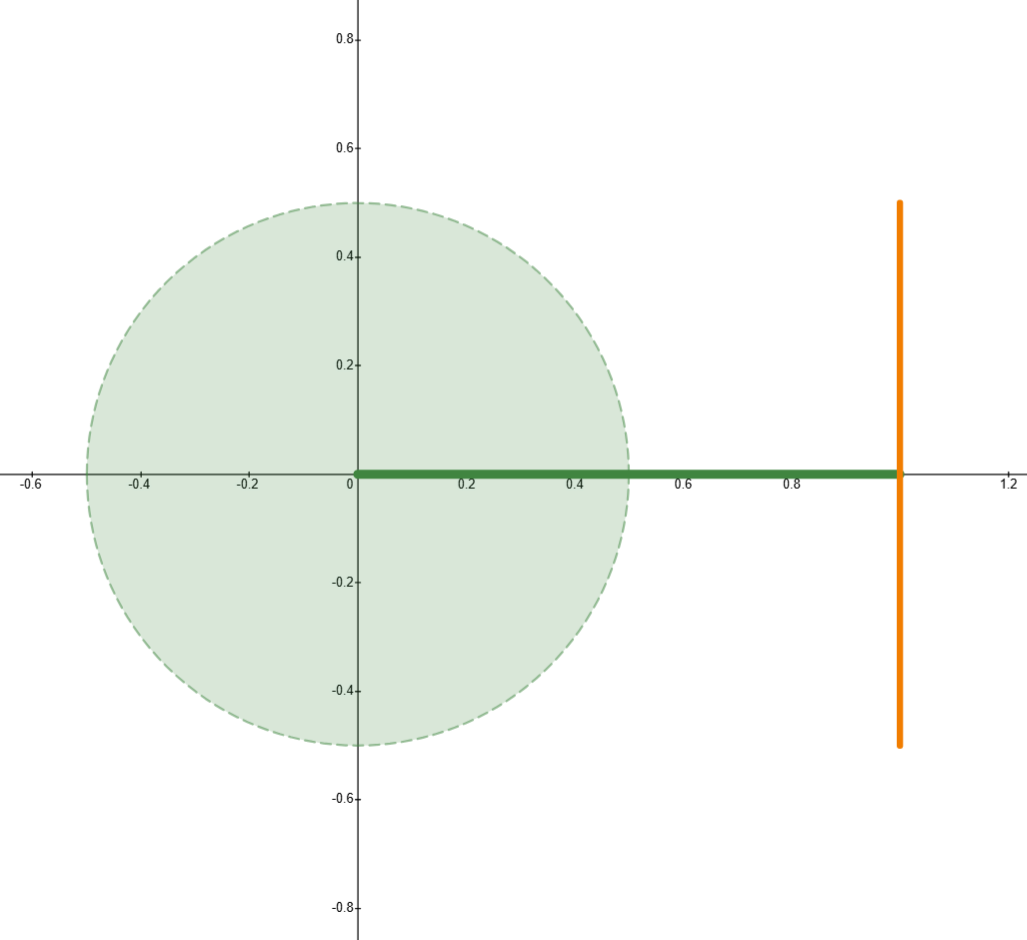}
    \caption{The cardinal 1-vine at time 1. }
    \label{fig:the-cardinal-1vine-at-time-1}
\end{figure}

The vine fully envelops the wall at time 2, at which point it sends out its tendril once more and encounters another wall of height 1 a distance \(\delta>0\) from the first. Each wall the vine envelops allows us enough time to build another.

\begin{figure}[H]
    \centering
    \includegraphics[width=1.0\linewidth]{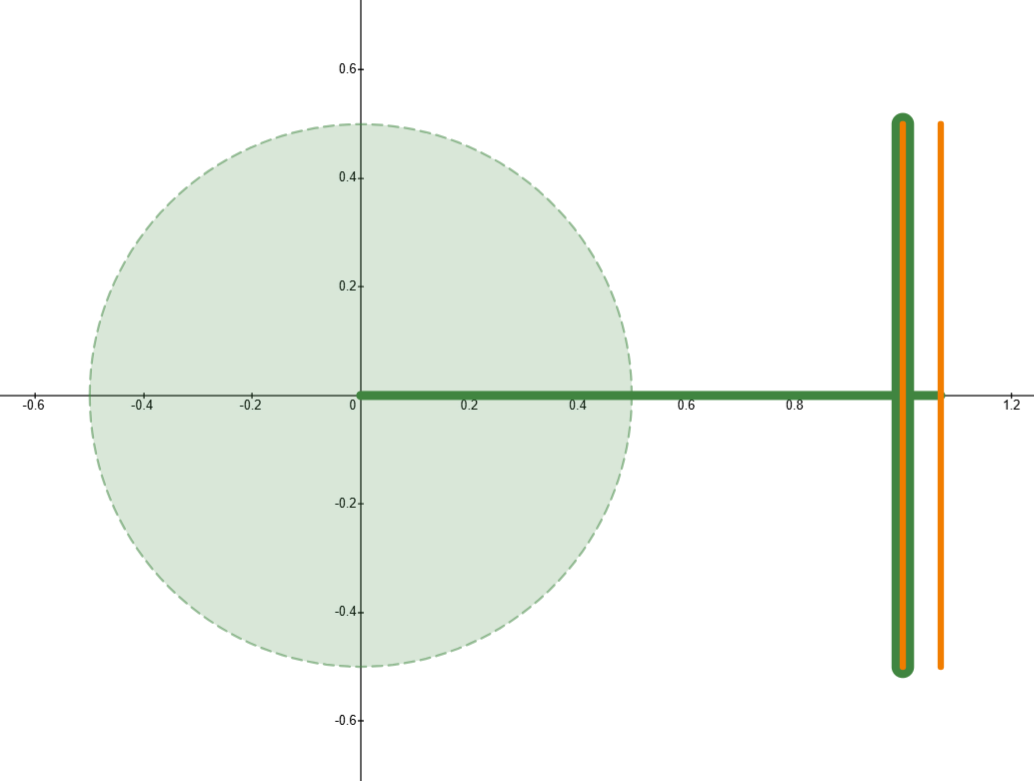}
    \caption{The cardinal 1-vine at time 2+\(\delta\). }
    \label{fig:the-cardinal-1vine-at-time-2}
\end{figure}

At time \(T+(T-2)\delta > T\), the vine has thus moved \(1+(T-2)\delta\), and lies entirely within an \(x\) range of \(1+T\delta\) and a \(y\) range of \(1+2\delta\). This, therefore, can be contained in a square of perimeter \(4+(2T+2)\delta\).

\begin{figure}[H]
    \centering
    \includegraphics[width=1.0\linewidth]{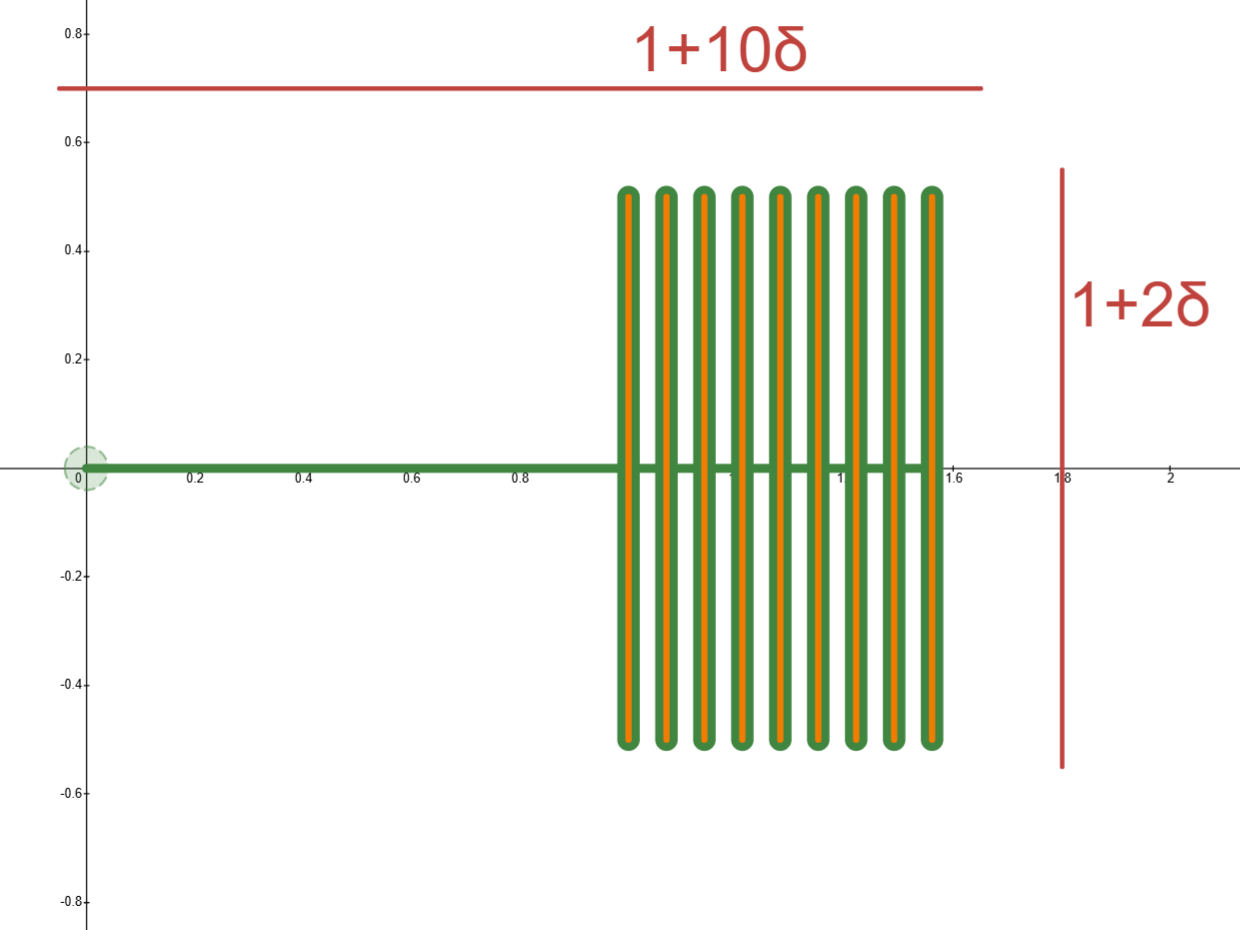}
    \caption{The cardinal 1-vine at time 10+8\(\delta\). Note that we can arbitrarily shrink the exclusion zone to fit within our parameters.}
    \label{fig:the-cardinal-1vine-at-time-10-plus-8delta}
\end{figure}

With \(\epsilon T\) length of wall, strictly less than the \(\epsilon T + \epsilon\delta(T-2)\) length of wall we can construct at this time, we see that 
\[4+(2T+2)\delta \leq \epsilon T\]
\[4+2\delta \leq \epsilon T-2\delta T\]
\[4+2\delta \leq (\epsilon -2\delta) T\]
so if we let \(\delta = \epsilon/4\), we have 
\[4+2\frac{\epsilon}{4} \leq \frac{\epsilon}{2}  T\]
\[8+\frac{\epsilon}{2} \leq \epsilon  T\]
\[8 \leq \epsilon  \left(T-\frac{1}{2}\right)\]
\[\frac{8}{\epsilon}+\frac{1}{2}\leq T\]
So if we construct \(T\) such walls for \(T\geq 8/\epsilon+3\), we will have saved enough wall to contain the vine.
\end{proof}

\begin{figure}[H]
    \centering
    \includegraphics[width=1\linewidth]{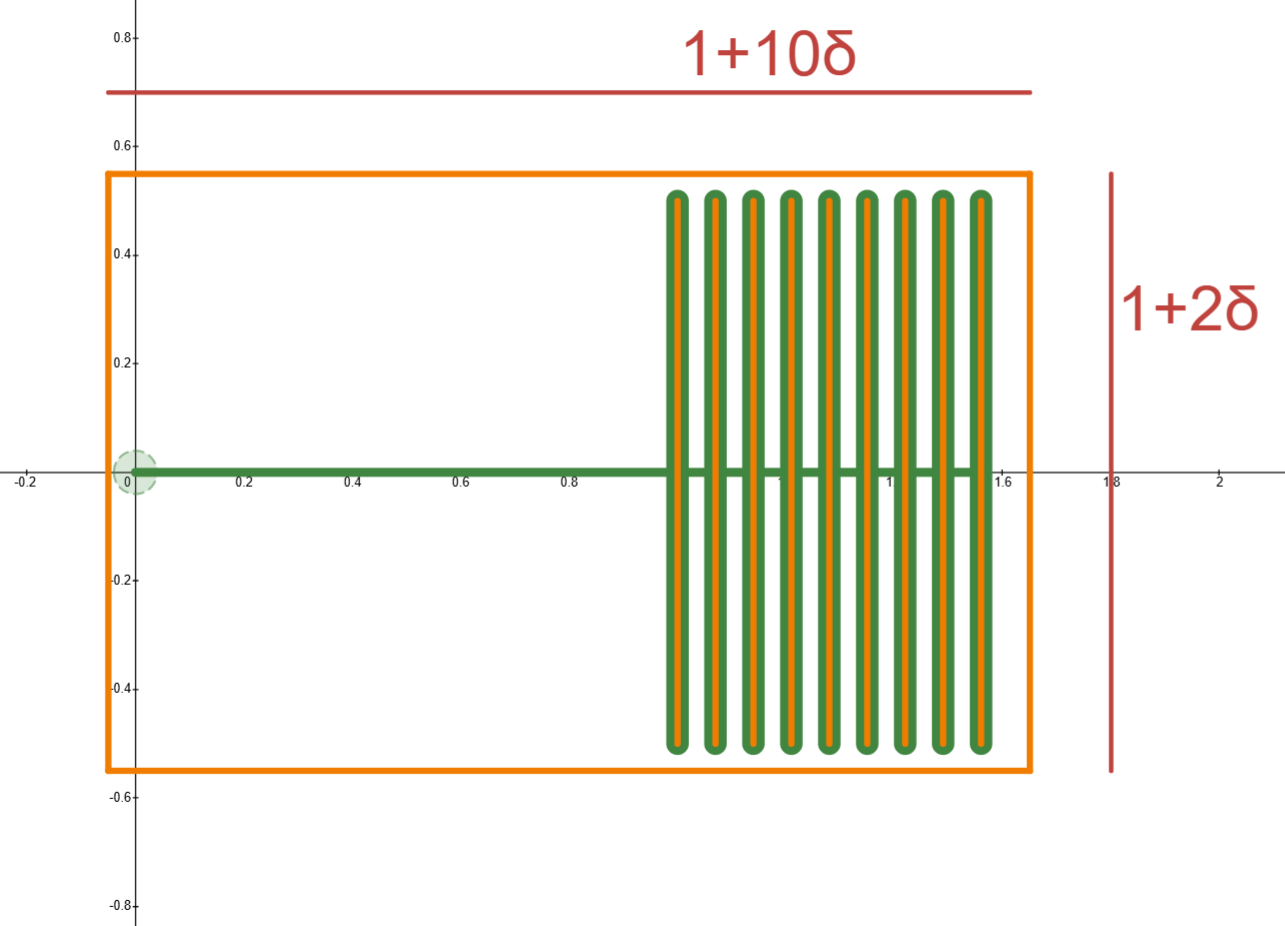}
    \caption{The cardinal 1-vine captured at time 10+8\(\delta\). In this diagram, \(\lambda \approx 1.6\).}
    \label{fig:the-cardinal-1vine-captured-at-time}
\end{figure}

\begin{proof}[Proof of necessity]
Assume \(\lambda=1\). By Lemma~\ref{lma:wall-reqs-contain} and Lemma~\ref{lma:wall-reqs-ray-motion}, due to the existence of the exclusion zone, the required length of non-vine covered surface to contain the vine, \(S_F\), must be strictly greater than twice the distance of the furthest point reached by the vine along the positive \(x\) axis at any given time \(t\). Therefore, if the time the tendril has freely traveled is \(T_F\), giving distance traveled along the ray of \(V_F=T_F\), we have \(S_F > 2V_F\). 

And, of course, if the tendril is not currently freely traveling, then the vine must be creeping along a surface. By Lemma~\ref{lma:creeping-rays} and Lemma~\ref{lma:other-wall-reqs-rays}, for every second the tendril spends halted, until such time as the vine is contained, there are at least two fronts, each of which must be covering a distance of one unit per second of wall surface. 

As such, the length of covered surface, \(S_C\), is at least twice the distance that a vine front would have crept over while the tendril was halted, \(V_C\), which is the same as the time that the tendril was halted, \(T_C\). This gives us \(V_C=T_C\) and \(S_C\geq 2V_C\). 

Finally, as the tendril is always halted or free, \(T=T_C+T_F\). 

Therefore, 
\[2W=S_F+S_C>2V_F+2V_C=2(V_F+V_C)=2(T_F+T_C)=2T\]
where \(W\) is the total length of wall (noting that a wall has two surfaces, and thus twice the surface length as wall length) and \(T\) is the time of containment, if such a containment is in fact possible. This gives us that 
\[\lambda T= 1 * T = T\geq W>T\]
and thus a contradiction.
\end{proof}

\begin{figure}[H]
    \centering
    \includegraphics[width=0.9\linewidth]{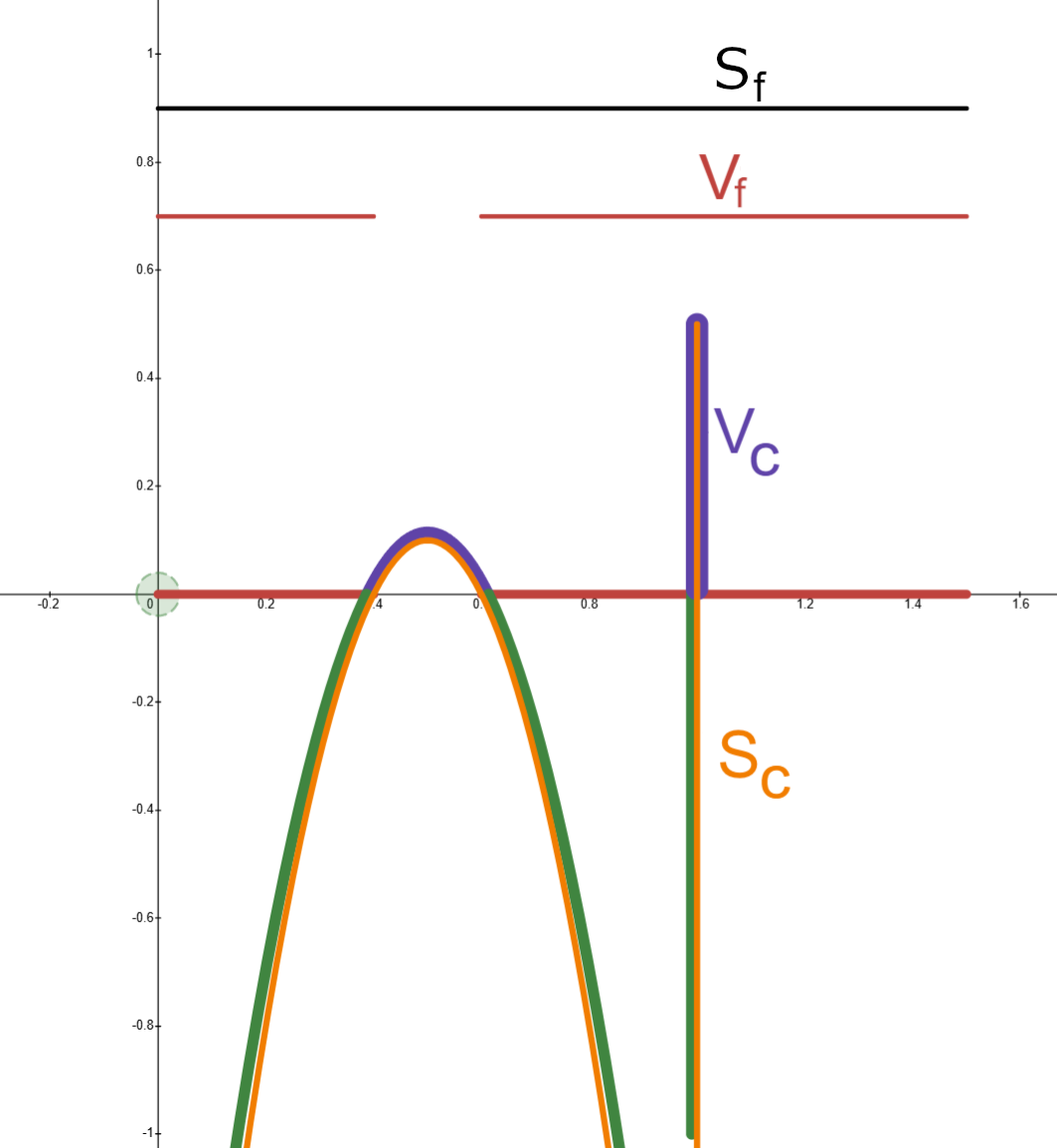}
    \caption{The cardinal 1-vine with measurements. The distance it travels to the east requires at least \(S_F\) (in black) length of uncovered wall, which we see must be more than twice the length of \(V_F\) (in red). Meanwhile, the length that the tendril creeps along during this motion, \(V_C\) (in purple) requires at least twice as much covered wall, which is \(S_C\) (in orange). If wall is only covered on one side, then it is orange, though it contributes its length once to \(S_F\) and once to \(S_C\), as opposed to twice to either one.}
    \label{fig:the-cardinal-1vine-with-measurements-the}
\end{figure}

We now move on to the cardinal 2-vine, and our improvement to the lower bound.

\begin{theorem}
\label{thm:cardinal-2-vine}
The cardinal \(2\)-vine has critical \(\lambda\) equal to \(1.5\): for \(\lambda>1.5\) it can be contained, and for \(\lambda\leq 1.5\) it cannot be. Consequently, the blob cannot be contained for any \(\lambda\leq 1.5\); that is, \(\lambda_C\geq 1.5\), improving on Theorem~\ref{thm:cardinal-1-vine}.
\end{theorem}

\begin{proof}[Proof of sufficiency]
Assume \(\lambda=1.5+\epsilon\) for \(\epsilon > 0\). Assume the vine moves directly west and east. We begin by constructing a vertical wall from (-1,1) to (-1,\(-\delta\)), connected to a horizontal wall which extends from \((-1,-\delta)\) to \((-0.5-\delta,-\delta)\).

At time 1, the west tendril hits our wall, and begins to creep around it. As it does so, we continue the horizontal wall east at rate \(1+\epsilon\). 

\begin{figure}[H]
    \centering
    \includegraphics[width=1.0\linewidth]{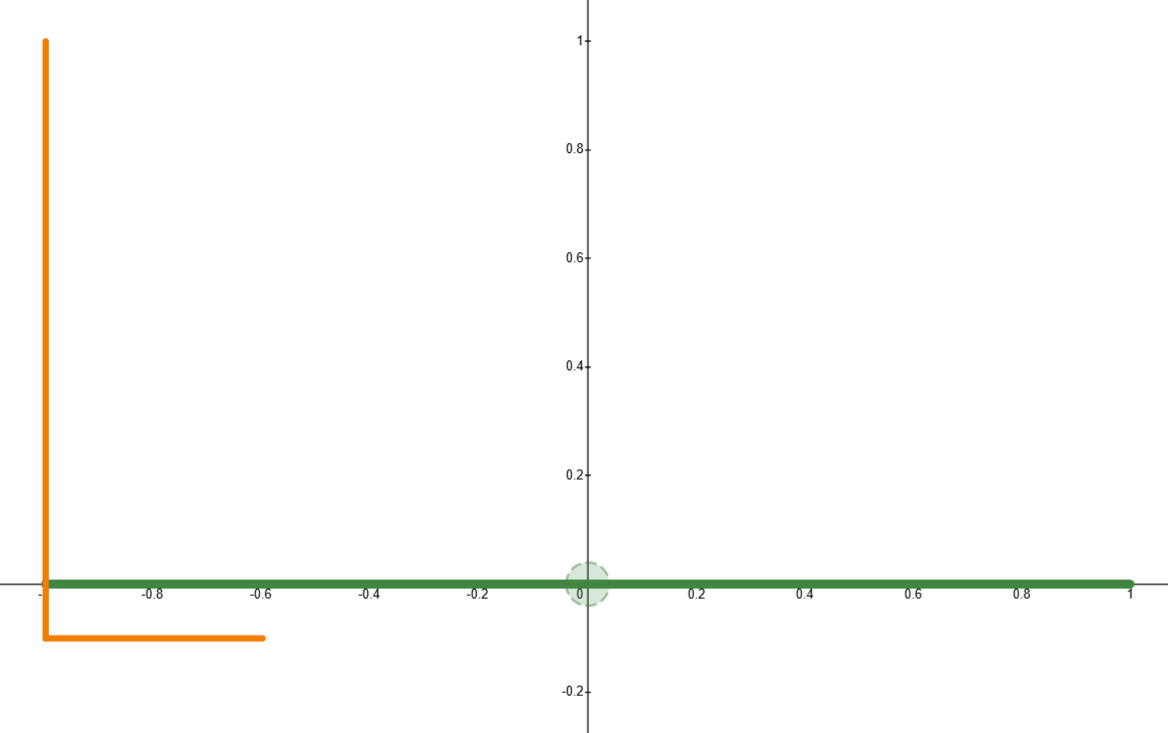}
    \caption{The cardinal 2-vine at time 1.}
    \label{fig:the-cardinal-2vine-at-time-1}
\end{figure}

At time 3, the west tendril would hit the x axis again, but we have already drawn a wall of length 1 on that point at an arbitrarily small angle to the previous vertical wall, diverting the tendril on another journey of length 2 in a manner that can be maintained indefinitely with an effective \(\lambda\) of 0.5.

\begin{figure}[H]
    \centering
    \includegraphics[width=1.0\linewidth]{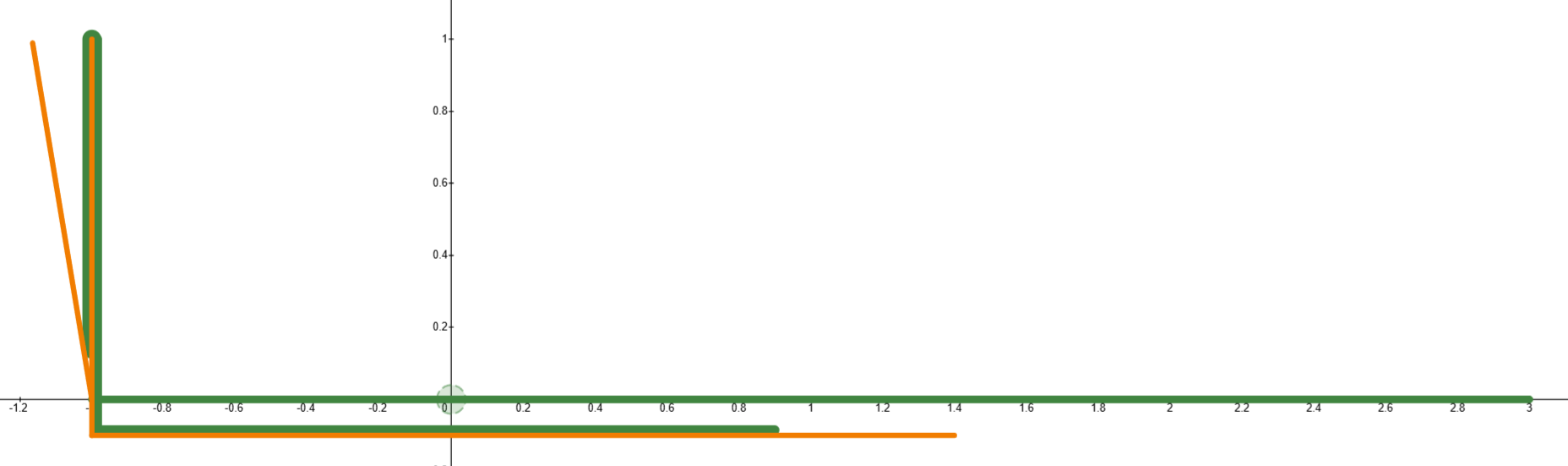}
    \caption{The cardinal 2-vine at time 3.}
    \label{fig:the-cardinal-2vine-at-time-3}
\end{figure}

We then catch up to the east-moving tendril, eventually, as we are building our wall east at a rate of \(\epsilon\) faster, and when we do, we return this wall, vine creeping along the interior, all the way to the other side where we capture it.
\end{proof}

\begin{figure}[H]
    \centering
    \includegraphics[width=1.0\linewidth]{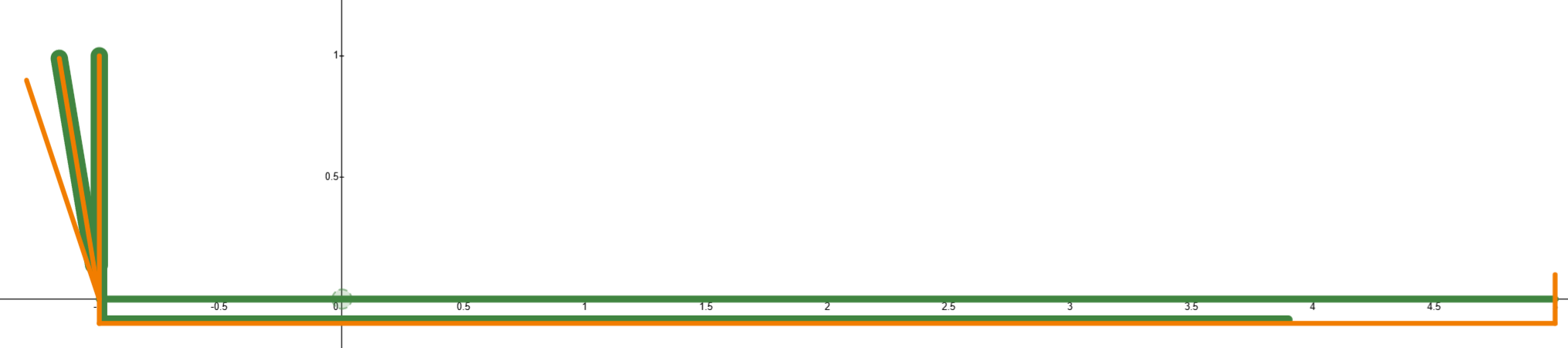}
    \caption{The cardinal 2-vine having both tendrils contained}
    \label{fig:the-cardinal-2vine-having-both-tendrils}
\end{figure}

\begin{figure}[H]
    \centering
    \includegraphics[width=1.0\linewidth]{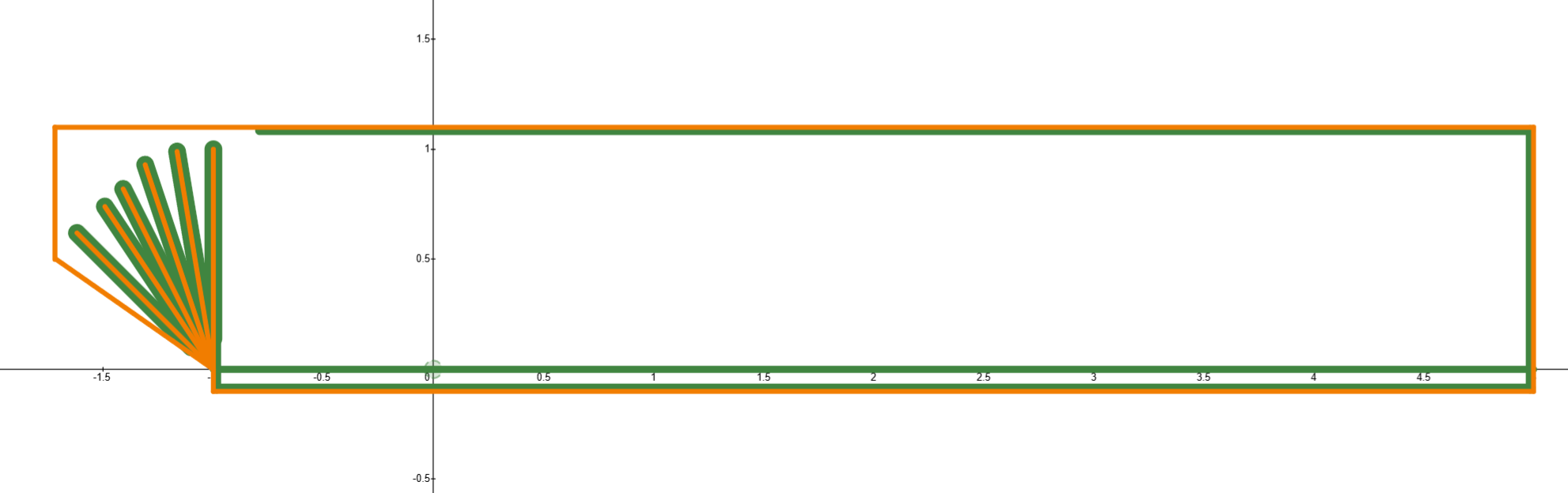}
    \caption{The cardinal 2-vine entirely contained.}
    \label{fig:the-cardinal-2vine-entirely-contained}
\end{figure}

\begin{proof}[Proof of necessity]
Assume \(\lambda=1.5\). We appear to have the same equation as before, though now with two tendrils, one might be forgiven for thinking we have 
\[2W=S_F+S_C>4V_F+4V_C=2(2V_F+2V_C)=2(2T_F+2T_C)=4T\]
as there are two tendrils active at any one time. However, this is not accurate: if both tendrils are currently halted, meaning that the last points at which they struck walls and halted must both have created points for the vine to creep, the vine may only be creeping at two points instead of four if the two of those creeping points collided thus causing the creeping to overlap. Lemma~\ref{lma:creeping-rays} and Lemma~\ref{lma:other-wall-reqs-rays} guarantee that we have at least \textit{two} fronts when both tendrils are halted, not at least four. 

Note that, by the above lemmas, should both tendrils be halted and there are no points at which the vine is creeping, then the entire surface struck by both tendrils must be covered by the vine. As this cannot include any more extreme points on the \(x\)-axis, as that would violate our assumption that the creeping arose from the last points at which the tendrils struck walls, the vine must be kept internal to those walls and thus the vine is currently contained. This cannot happen until the above inequality is satisfied, and therefore, to be able to arrive at the point of containment, the vine must at some point have both tendrils halted with the creeping from the last points of the tendrils overlapping.

\begin{figure}[H]
    \centering
    \includegraphics[width=0.4\linewidth]{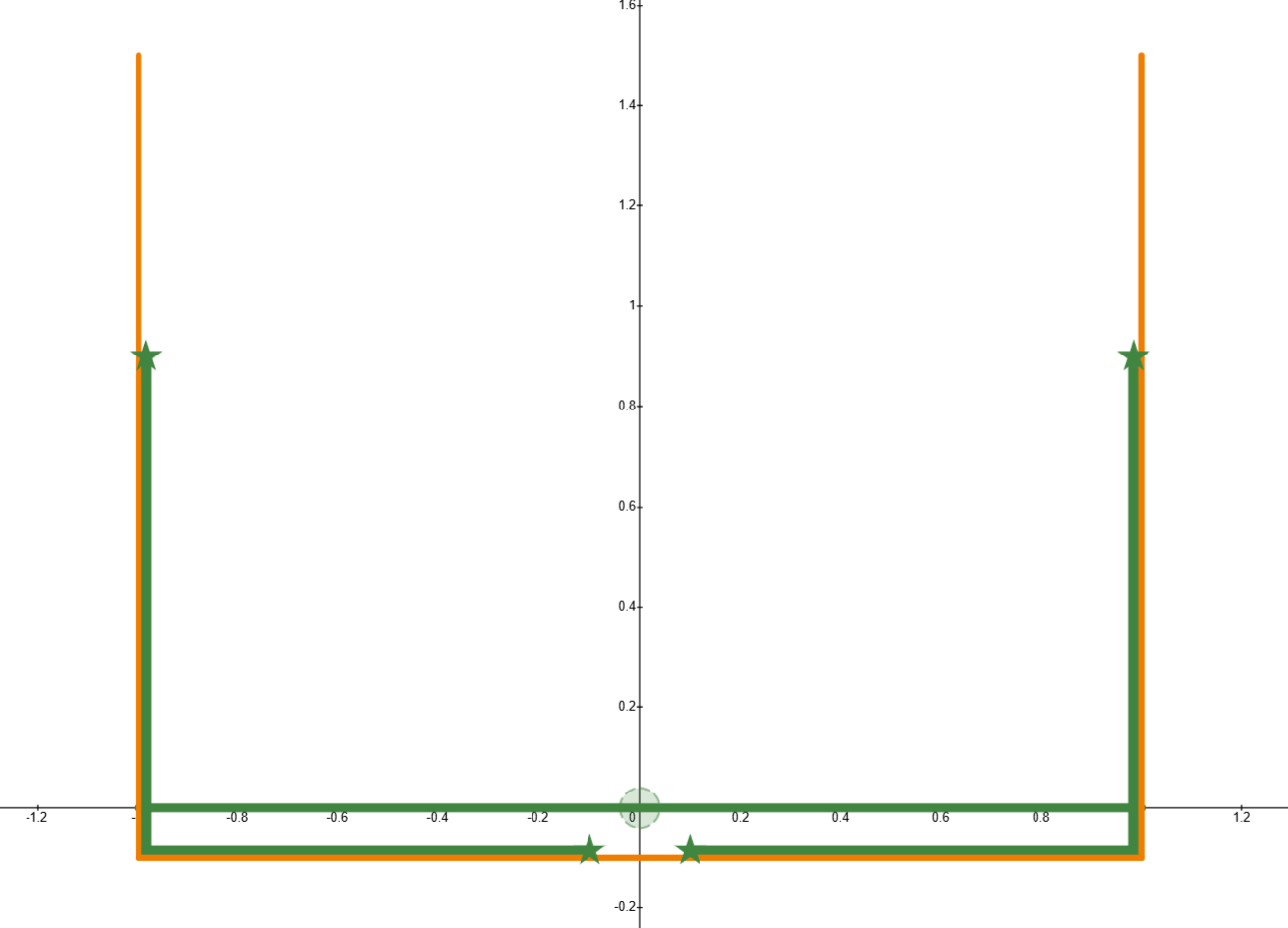}
    \includegraphics[width=0.4\linewidth]{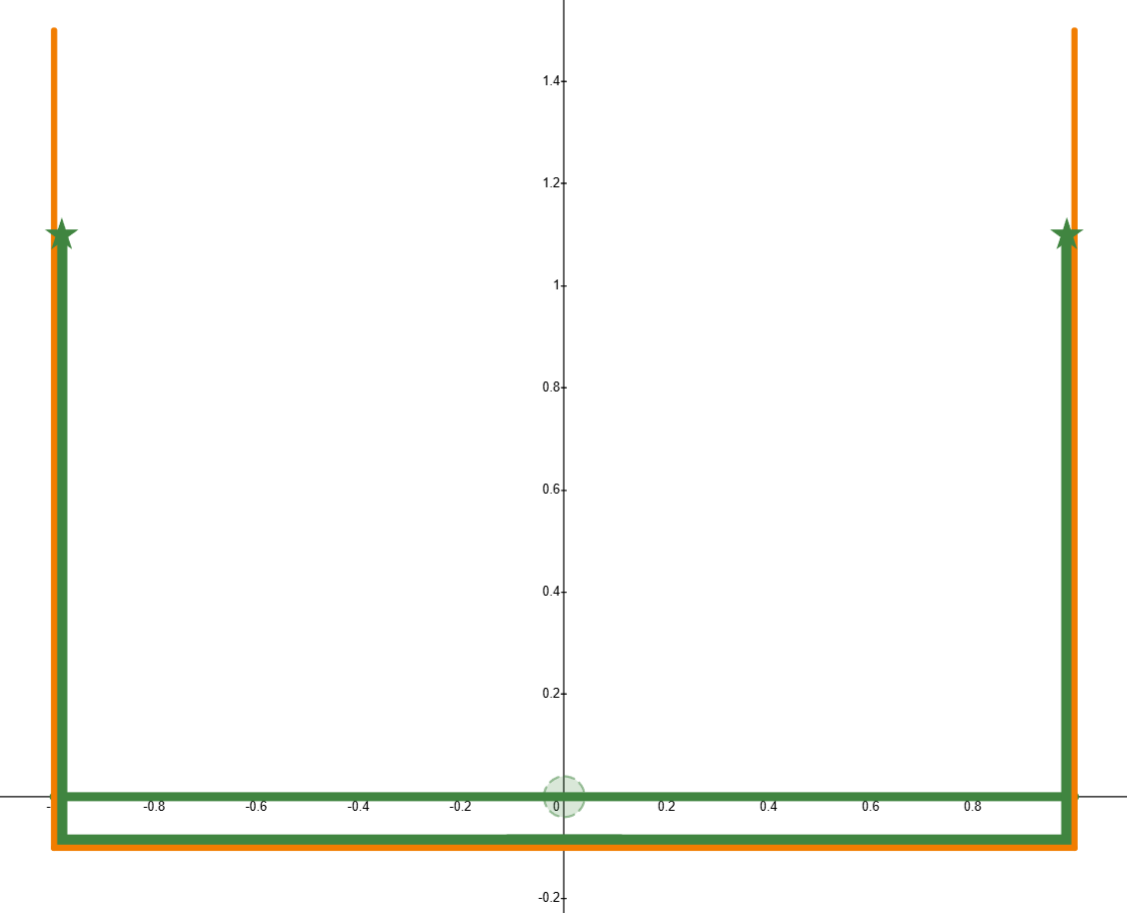}
    \caption{The vine creeping at four points and then two points, marked with stars.}
    \label{fig:the-vine-creeping-at-four-points}
\end{figure}

In other words, the true inequality is 
\begin{align*}
2W&=S_F+S_C\\
&=S_F+S_{1,C}+S_{2,C}\\
&>2V_{1,F}+2V_{2,F}+2V_{1,\not2,C}+2V_{1,2,C}^{(4)}+V_{1,2,C}^{(2)}+2V_{2,\not1,C}+2V_{1,2,C}^{(4)}+V_{1,2,C}^{(2)}\\
&=2T_{1,F}+2T_{2,F}+2T_{1,\not2,C}+2T_{1,2,C}^{(4)}+T_{1,2,C}^{(2)}+2T_{2,\not1,C}+2T_{1,2,C}^{(4)}+T_{1,2,C}^{(2)}\\
&=2(T_{1,F}+2T_{1,\not2,C}+2T_{1,2,C}^{(4)}+2T_{1,2,C}^{(2)})+2(T_{2,F}+2T_{2,\not1,C}+2T_{1,2,C}^{(4)}+2T_{1,2,C}^{(2)})-2T_{1,2,C}^{(2)}\\
&=2T+2T-2T_{1,2,C}^{(2)}\\
&=4T-2T_{1,2,C}^{(2)}
\end{align*}
where \(V_{i,F}\) is the distance moved freely over by tendril \(i\),  \(V_{i,\not j, C}\) is the distance that two fronts would creep over the time that tendril \(i\) was halted and tendril \(j\) was not, \(V_{i, j, C}^{(4)}\) is the distance that four fronts would creep over the time that tendrils \(i\) and \(j\) were both halted, \(V_{i, j, C}^{(2)}\) is the distance that two fronts would creep over the time that tendrils \(i\) and \(j\) were both halted, \(T_{i,F}\) is the total time tendril \(i\) was not halted,  \(T_{i,\not j, C}\) is the total time that tendril \(i\) was halted and tendril \(j\) was not, \(T_{i, j, C}^{(4)}\) is the total time that tendrils \(i\) and \(j\) were both halted and there were at least four fronts, and \(T_{i, j, C}^{(2)}\) is the total time that tendrils \(i\) and \(j\) were both halted and there were only two fronts. Naturally, \(T=T_{i,F}+T_{i,\not j, C}+T_{i, j, C}^{(4)}+T_{i, j, C}^{(2)}\), as over the course of the full time, tendril \(i\) is either halted or not halted, and if tendril \(i\) is halted, then tendril \(j\) is either halted or not halted, and if tendril \(j\) is halted, then either there were 2 fronts active or there were more than two fronts. 

And, as \(\lambda T\geq W\), we have 
\[2\lambda T > 4T-2T_{1,2,C}^{(2)}\] 
\[\lambda > 2-\frac{T_{1,2,C}^{(2)}}{T}\] 

So if \(\lambda\leq 1.5\), \(\frac{T_{1,2,C}^{(2)}}{T}> 0.5 \) and thus \(T_{1,2,C}^{(2)}>0\). Therefore, there must be some time where both tendrils are halted and there are only two fronts. Therefore, there must be some first such time.

We now analyze what is required for the first time there are two halted tendrils with only two fronts. 

Suppose that this is the case at some time \(T\), and it is not the case for any time \(<T\). The west tendril hit the wall it's currently creeping on at distance \(D_1\) and the east tendril hit at \(D_2\). If the west tendril hit the wall at a time \(T_1\), then \(T_1\leq D_1+C_1\) where \(C_1=T_{1,\not2,C}+T_{1,2,C}\) is the time the tendril spent creeping (the notation is simplified here from the prior notation for improved clarity of both the diagrams and the algebra). An example of this can be seen in Figure~\ref{fig:the-cardinal-2vine-the-distances-and}, where the left tendril skips over a portion of the \(x\)-axis by creeping along the top of the wall. 

\begin{figure}[H]
    \centering
    \includegraphics[width=0.8\linewidth]{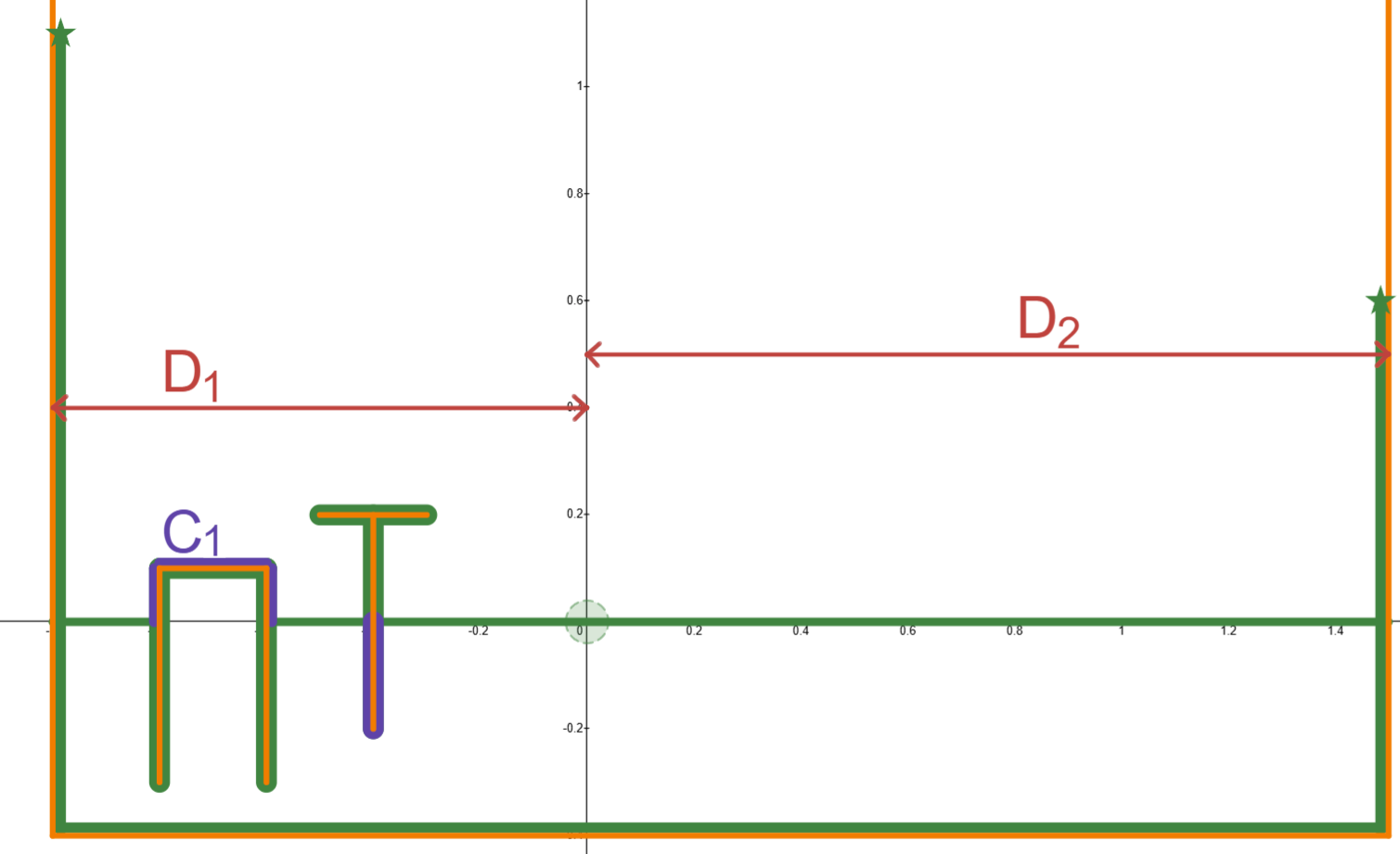}
    \caption{The cardinal 2-vine. The distances \(D_1\) and \(D_2\) are in red, while \(C_1\) is in purple. Note that, here, \(T_1<D_1+C_1\) as the tendril does skip part of the \(x\)-axis while creeping.}
    \label{fig:the-cardinal-2vine-the-distances-and}
\end{figure}

As shown above, this requires a length of at least \(2C_1\) surface being touched prior to \(T_1\). Similarly, if the east tendril hit the wall at time \(T_2\), then \(T_2\leq D_2+C_2\) with a length of \(2C_2\) surface being necessarily touched by the vine before time \(T_2\). 

Note that, as there is no time before \(T\) such that both tendrils were halted and there were only two fronts, and \(T_1,T_2\leq T\), if both tendrils were halted prior to \(T\), there must have been at least four fronts. Therefore, the total length of surface touched during the times \(C_1\) and \(C_2\) must be at least \(2C_1+2C_2\), with no crossover.

\begin{figure}[H]
    \centering
    \includegraphics[width=1.0\linewidth]{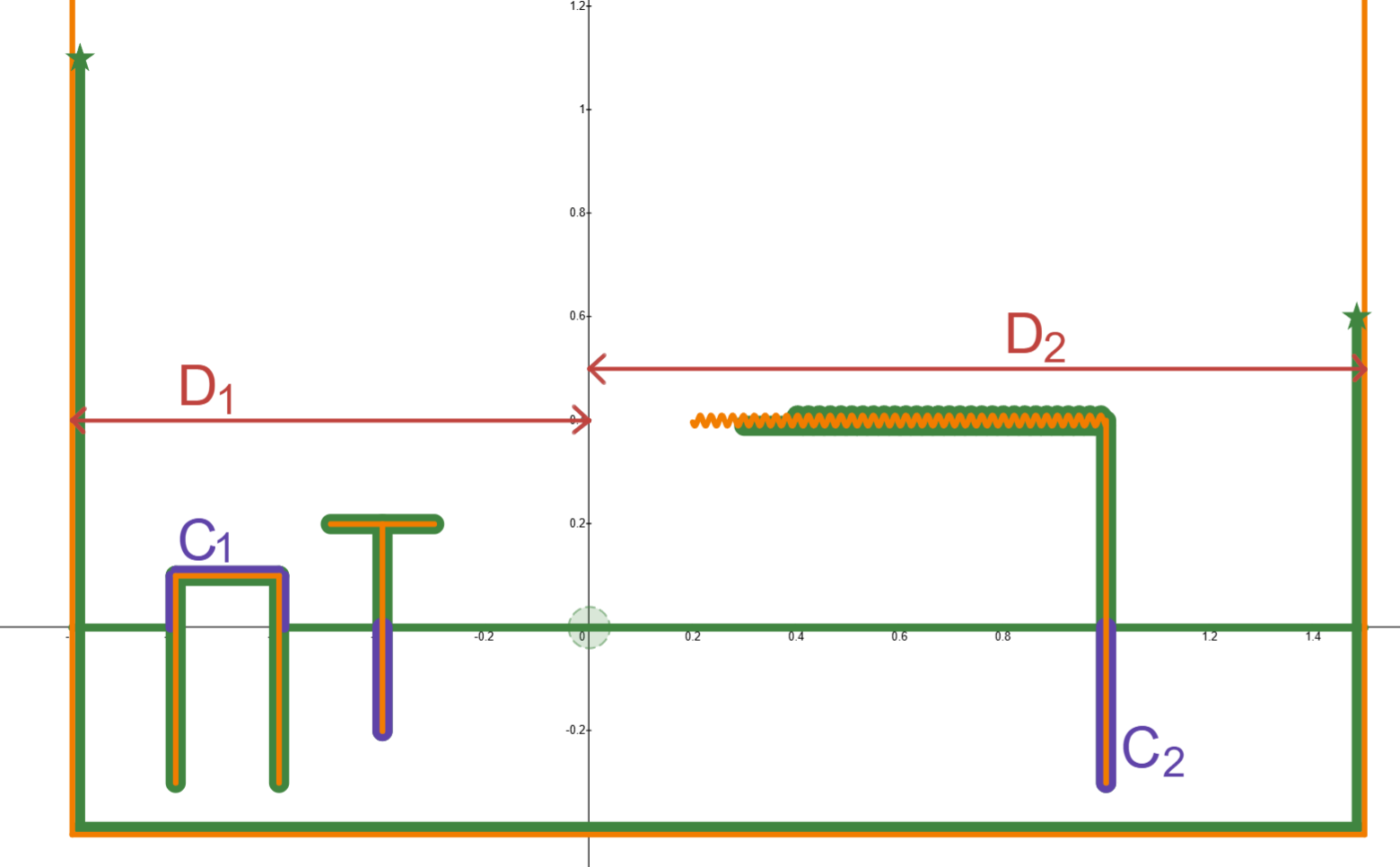}
    \caption{The cardinal 2-vine. The distances \(D_1\) and \(D_2\) are in red, while \(C_1\) and \(C_2\) are in purple. Also note that, while not all of the wall along the east tendril's path is fully covered, the covered surface must be of length at least \(C_2\) on either side}
    \label{fig:the-cardinal-2vine-the-distances-and-2}
\end{figure}

Now, as the two fronts need to touch, the vine must creep a total distance of at least \(D_1+D_2\) along surfaces connecting the two points of the wall that the east and west tendrils hit. In order to make progress, and not backtrack, this requires the front to creep along the surface only in one direction. The length that a front must creep from the west side to the meeting point is \(D_3\), and the length that a front must creep from the east side to the meeting point is \(D_4\), which thus requires a unique contiguous stretch of walls touched by the vine only on one side from the point that the west tendril last was halted at to the point that the east tendril was last halted at. This therefore takes time at most \(D_3\) for the west front and at most \(D_4\) for the east front, as the slowest that lengths \(D_3\) and \(D_4\) can be crept along is 1 unit per second.

We let the time that a front is not moving along these one-sided walls on the west side be \(C_3\), and likewise we let the time that a front is not moving along these one-sided walls on the east side be \(C_4\). During these times, as both tendrils are halted and there must be at least 4 fronts active, the surface covered by the at least 2 fronts that will be returning to the one-sided walls must be at least \(C_3\) and \(C_4\), respectively.

For the west and east tendrils, respectively, this takes wall surface lengths \(D_3\) and \(D_4\) touched by the vine, lengths \(D_3\) and \(D_4\) of surface not touched by the vine, lengths \(C_3\) and \(C_4\) of surface for creeping, and therefore times \(T_3 \leq  D_3+C_3\) and \(T_4\leq D_4+C_4\), noting that \(T_1+T_3=T_2+T_4=T\), \(D=D_1+D_2< D_3+D_4=D'\) as \(D_3\) and \(D_4\) must make up the horizontal distance gained by \(D_1\) and \(D_2\) while also moving around the tiny exclusion zone around the origin, and at time \(T\), a length of \(2C_1+2C_2+C_3+C_4\) of surface being necessarily touched by the vine prior to time \(T\). 

\begin{figure}[H]
    \centering
    \includegraphics[width=1.0\linewidth]{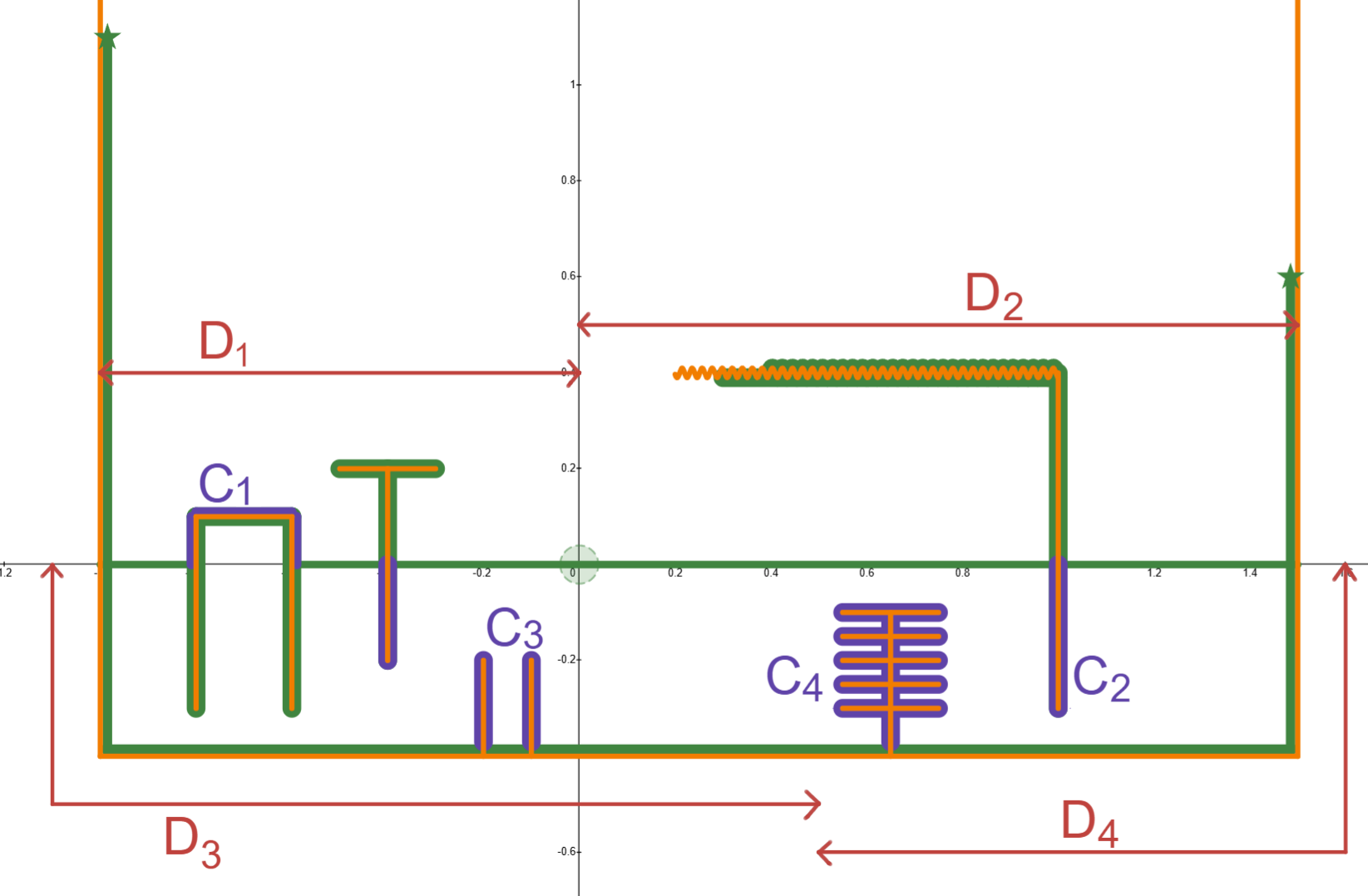}
    \caption{Observe that, as the vine is creeping only on one side of the wall on the bottom, all lengths are exact and thus \(T_3\) and \(T_4\) are equalities.}
    \label{fig:observe-that-as-the-vine-is}
\end{figure}

Note that \(C_3\) and \(C_4\) are not multiplied by two here, unlike elsewhere, as there is no ``other" front created. This brings our wall surface total to \(2C_1+2C_2+C_3+C_4+2D_3+2D_4\).

Moreover, in order to prevent freeing the east and west tendrils, we must occupy that other half of the creeping. The best we can do by Lemma~\ref{lma:total-creeping} is lengths of \(T_3\) surface on the west and \(T_4\) surface on the east, as we need to prevent the freeing of the east and west tendrils for \(T_3\) and \(T_4\) seconds, respectively. 

\begin{figure}[H]
    \centering
    \includegraphics[width=1\linewidth]{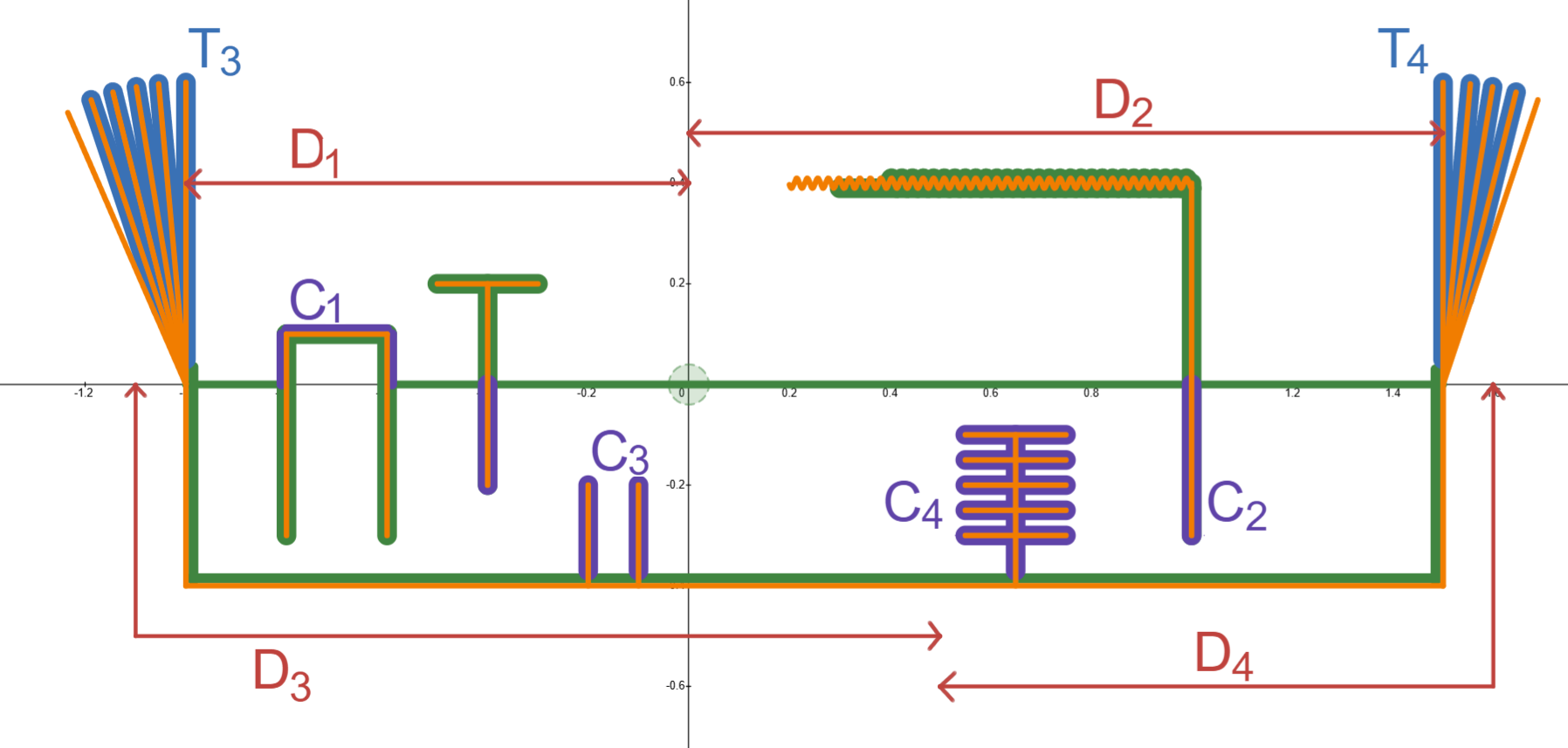}
    \caption{The full diagram with all labels. }
    \label{fig:the-full-diagram-with-all-labels}
\end{figure}

In total, this gives us a required length of wall at time \(T\) of 
\begin{align*}
2W &> 2C_1+2C_2+C_3+C_4+2D_3+2D_4+T_3+T_4\\
W &> C_1+C_2+\frac{C_3+C_4}{2}+D_3+D_4+\frac{T_3+T_4}{2}\\
 &= \frac{C_1+C_2}{2}+\frac{C_1+C_2+C_3+C_4}{2}+\frac{D_3+D_4+D_3+D_4}{2}+\frac{T_3+T_4}{2}\\
 &> \frac{C_1+C_2}{2}+\frac{C_1+C_2+C_3+C_4}{2}+\frac{D_1+D_2+D_3+D_4}{2}+\frac{T_3+T_4}{2}\\
 &= \frac{C_1+C_2}{2}+\frac{D_1+D_2}{2}+\frac{C_1+C_2+C_3+C_4}{2}+\frac{D_3+D_4}{2}+\frac{T_3+T_4}{2}\\
 &\geq \frac{T_1+T_2}{2}+\frac{C_1+C_2+C_3+C_4}{2}+\frac{D_3+D_4}{2}+\frac{T_3+T_4}{2}\\
 &= \frac{T_1+T_2}{2}+\frac{T_3+T_4}{2}+\frac{C_1+C_2+C_3+C_4}{2}+\frac{D_3+D_4}{2}\\
 &= \frac{2T}{2}+\frac{C_1+C_2+C_3+C_4}{2}+\frac{D_3+D_4}{2}\\
 &= T+\frac{C_1+C_2+C_3+C_4}{2}+\frac{D_3+D_4}{2}\\
\end{align*}
Now, as \(T_1+T_2+T_3+T_4=2T\), we know that \(T_1+T_2\geq T\) or \(T_3+T_4\geq T\). If \(T_1+T_2\geq T\), then as \(D_1+D_2<D_3+D_4\), 

\begin{align*}
W &> T+\frac{C_1+C_2+C_3+C_4}{2}+\frac{D_3+D_4}{2}\\
 &> T+\frac{C_1+C_2+C_3+C_4}{2}+\frac{D_1+D_2}{2}\\
 &> T+\frac{C_3+C_4}{2}+\frac{T_1+T_2}{2}\\
 &\geq T+\frac{C_3+C_4}{2}+\frac{T}{2}\\
 &\geq T+\frac{T}{2}=\frac{3T}{2}\\
\end{align*}
On the other hand, if \(T_3+T_4\geq T\), then 
\begin{align*}
W &> T+\frac{C_1+C_2+C_3+C_4}{2}+\frac{D_3+D_4}{2}\\
 &\geq T+\frac{C_1+C_2}{2}+\frac{T_3+T_4}{2}\\
 &\geq T+\frac{C_1+C_2}{2}+\frac{T}{2}\\
 &\geq T+\frac{T}{2}=\frac{3T}{2}\\
\end{align*}

and so \(\lambda > 3/2=1.5\).
\end{proof}

\begin{corollary}
\label{cor:main}
The blob cannot be contained for any \(\lambda\leq 1.5\); that is, \(\lambda_C\geq 1.5\).
\end{corollary}
\begin{proof}
The cardinal \(2\)-vine is a specific instance of the \(n\)-vine, which is a weakening of the blob (Lemma~\ref{lma:blob-weakening}), so Theorem~\ref{thm:cardinal-2-vine} applies directly.
\end{proof}

Unfortunately, for \(n\geq 2\), we can contain the cardinal \(n\)-vine with \(\lambda>1.5\) using the strategy above, whereby we contain two vines, keep one in measure with 0.5 units of wall per second, and with the remaining \(\lambda-0.1>1\) we slowly encircle the remaining vines. Adding more tendrils to the cardinal vine does not improve the bounds.

\begin{theorem}
\label{thm:cardinal-n-vine}
The cardinal \(n\)-vine has critical \(\lambda\) equal to \(1.5\) for \(n>2\). 
\end{theorem}
\begin{proof}[Proof]
Using a similar strategy to containing the cardinal 2-vine with \(\lambda>1.5\), we first block one tendril. Occupying one of its fronts with a baffle, requiring 0.5 walls per second, we then use the remaining \(\lambda-0.5>1\) walls per second to chase down and catch an adjacent tendril. Now that we have done so, once the internal fronts collide, we again have two fronts: one baffled at 0.5 walls per second, and one outsped at \(\lambda-0.5>1\) walls per second. Using this, we can catch up to and contain each tendril in turn, and then eventually finish our loop to contain the cardinal \(n\)-vine.
\end{proof}

Further work is needed to push the lower bound past 1.5.

\end{document}